\documentclass[aop,MSNbibl,citesort,dvips]{arximspdf}
\usepackage{graphicx}

\doi{10.1214/11-AOP644}
\volume{40}
\issue{3}
\pubyear{2012}
\firstpage{1167}
\lastpage{1211}

\makeatletter
\newcommand{\reff}[1]{(\ref{#1})}

\newtheorem{theo}{Theorem}[section]
\newtheorem{cor}[theo]{Corollary}
\newtheorem{prop}[theo]{Proposition}
\newtheorem{lem}[theo]{Lemma}
\newproclaim{defi}[theo]{Definition}
\newproclaim{rem}[theo]{Remark}

\newcommand{\cb}{{\mathcal B}}
\newcommand{\cd}{{\mathcal D}}
\newcommand{\cf}{{\mathcal F}}
\newcommand{\cg}{{\mathcal G}}
\newcommand{\ch}{{\mathcal H}}
\newcommand{\ci}{{\mathcal I}}
\newcommand{\cj}{{\mathcal J}}
\newcommand{\cl}{{\mathcal L}}
\newcommand{\cn}{{\mathcal N}}
\newcommand{\cm}{{\mathcal M}}
\newcommand{\crr}{{\mathcal R}}
\newcommand{\ct}{{\mathcal T}}
\newcommand{\cw}{{\mathcal W}}

\newcommand{\E}{{\mathbb E}}
\newcommand{\N}{{\mathbb N}}
\newcommand{\Nb}{\bar{\mathbb N}}
\renewcommand{\P}{{\mathbb P}}
\newcommand{\Pb}{{\bar{\mathbb P}}}
\newcommand{\bP}{{\pmb{\mathbb P}}}
\renewcommand{\bP}{\mathbf{P}}
\newcommand{\bE}{{\pmb{\mathbb E}}}
\renewcommand{\bE}{{\mathbf E}}
\newcommand{\bN}{{\pmb{\mathbb N}}}
\renewcommand{\bN}{{\mathbf N}}
\newcommand{\R}{{\mathbb R}}
\newcommand{\M}{\mathbb M}

\newcommand{\rP}{\mathrm{P}}
\newcommand{\rE}{\mathrm{E}}
\newcommand{\ind}{{\mathbf1}}

\newcommand{\Supp}{\operatorname{Supp}}

\newcommand{\lb}{[\![}
\newcommand{\rb}{]\!]}
\makeatother

\begin{document}
\begin{frontmatter}

\title{A continuum-tree-valued Markov process\thanksref{TT1}}
\runtitle{A continuum-tree-valued Markov process}

\begin{aug}
\author[A]{\fnms{Romain} \snm{Abraham}\corref{}\ead[label=e1]{romain.abraham@univ-orleans.fr}}
\and
\author[B]{\fnms{Jean-Fran\c{c}ois} \snm{Delmas}\ead[label=e2]{delmas@cermics.enpc.fr}}
\runauthor{R. Abraham and J.-F. Delmas}
\affiliation{Universit\'{e} d'Orl\'{e}ans and Universit\'{e} Paris-Est}
\address[A]{Laboratoire MAPMO\\
 CNRS, UMR 6628\\
F\'{e}d\'{e}ration Denis Poisson, FR 2964\\
Universit\'{e} d'Orl\'{e}ans\\
B.P. 6759\\
45067 Orl\'{e}ans cedex 2\\
France\\
\printead{e1}} 
\address[B]{Universit\'{e} Paris-Est\\
\'{E}cole des Ponts, CERMICS\\
6-8
av. Blaise Pascal,
Champs-sur-Marne\\ 77455 Marne La Vall\'{e}e\\ France\\
\printead{e2}}
\end{aug}
\thankstext{TT1}{Supported in part by the ``Agence Nationale de
  la Recherche,'' ANR-08-BLAN-0190.}

\received{\smonth{4} \syear{2009}}
\revised{\smonth{7} \syear{2010}}

%
\begin{abstract}
We present a construction of a L\'{e}vy continuum random tree (CRT)
associated with a super-critical continuous state branching process
using the so-called exploration process and a Girsanov theorem. We
also extend the pruning procedure to this super-critical case. Let
$\psi$ be a critical branching mechanism. We set
$\psi_\theta(\cdot)=\psi(\cdot+\theta)-\psi(\theta)$. Let
$\Theta=(\theta_\infty,+\infty)$ or $\Theta=[\theta_\infty
,+\infty)$
be the set of values of~$\theta$ for which $\psi_\theta$ is a
conservative branching mechanism. The pruning procedure allows to
construct a
decreasing L\'{e}vy-CRT-valued Markov process
$(\ct_\theta,\theta\in\Theta)$, such that $\ct_\theta$ has branching
mechanism $\psi_\theta$. It is sub-critical if $\theta>0$ and
super-critical if $\theta<0$. We then consider the explosion time $A$
of the CRT: the smallest (negative) time $\theta$ for which
the continuous state branching process (CB) associated with $\ct
_\theta$
has finite total mass (i.e., the length of the excursion of the
exploration process that codes the CRT is finite). We describe the law
of $A$ as well as
the distribution of the CRT just after this explosion time. The CRT
just after explosion can be seen as a CRT conditioned not to be
extinct which is pruned with an independent intensity related to $A$.
We also study the evolution of the CRT-valued process after the
explosion time. This extends results from Aldous and Pitman on
Galton--Watson trees. For the particular case of the quadratic
branching mechanism, we show that after explosion the total mass of
the CB behaves like the inverse of a stable subordinator with index
$1/2$. This result is related to the size of the tagged fragment for
the fragmentation of Aldous's CRT.
\end{abstract}

%

\end{frontmatter}

\section{Introduction}\label{intro}

Continuous state branching processes (CB in short) are nonnegative
real valued Markov processes first introduced by Jirina \cite{jsbpcss} that
satisfy a branching property: the process $(Z_t,t\ge0)$
is a CB if its law when starting from $x+x'$ is equal to the law
of the sum of two independent copies of $Z$ starting respectively from
$x$ and $x'$. The law of such a process is characterized by the
so-called branching mechanism $\psi$ via its Laplace functionals.
The branching mechanism $\psi$ of a CB is given by
\[
\psi(\lambda)=\tilde\alpha\lambda
+\beta\lambda^2+\int_{(0,+\infty)}\pi(d\ell) \bigl[\mathrm
{e}^{-\lambda \ell}-1+\lambda\ell\ind_{\{\ell\le1\}} \bigr],
\]
where $\tilde\alpha\in\R$, $\beta\ge0$ and $\pi$ is a Radon
measure on
$(0,+\infty)$ such that $\int_{(0,+\infty)}(1\wedge
\ell^2)\pi(d\ell)<+\infty$.
The CB is said to be respectively sub-critical, critical,
super-critical when $\psi'(0)>0$, $\psi'(0)=0$ or
$\psi'(0)<0$. We will write (sub)critical for critical or sub-critical.
Notice that $\psi$ is smooth and strictly convex if $\beta>0$ or
$\pi\neq0$.

 It is shown in \cite{llsbp} that all these CBs can be obtained
as the limit of renormalized sequences of Galton--Watson processes. A
genealogical tree is naturally associated with a Galton--Watson process
and the question of existence of such a genealogical structure for CB
arises naturally. This question has given birth to the theory of continuum
random trees (CRT), first introduced in the pioneer work of Aldous
\cite{acrt1,acrt2,acrt3}. A continuum random tree (called L\'{e}vy
CRT) that codes the genealogy of a general (sub)critical branching
process has been constructed in \cite{lgljbplplfss,lgljbplpep} and
studied further in \cite{dlgrtlpsbp}. The main tool of this approach
is the so-called exploration process $(\rho_s, s\in\R^+)$,
where~$\rho_s$ is a measure on $\R^+$, which codes for the
CRT. For (sub)critical quadratic branching mechanism ($\pi=0)$, the
measure $\rho_s$ is just the Lebesgue measure over an interval
$[0,H_s]$, and the so-called height process $(H_s, s\in\R^+)$ is a
Brownian motion with drift reflected at $0$. In \cite{dhpscsbp}, a CRT
is built for super-critical quadratic branching mechanism using the Girsanov
theorem for Brownian motion.

We propose here a construction for general super-critical L\'{e}vy tree,
using the exploration process, based on ideas from \cite{dhpscsbp}. We
first build the super-critical tree up to a given level $a$. This tree
can be coded by an exploration process, and its law is absolutely
continuous with respect to the law of a~(sub)critical L\'{e}vy tree, whose
leaves above level $a$ are removed. Moreover, this family of processes
(indexed by parameter $a$) satisfies a compatibility property, and
hence there exists a projective limit which can be seen as the law of
the CRT associated with the super-critical CB. This construction enables
us to use most of the results known for (sub)critical CRT. Notice that
another construction of a L\'{e}vy CRT that does not make use of the
exploration process has been proposed in \cite{dwglt} as the limit, for
the Gromov--Hausdorff metric, of a sequence of discrete trees. This
construction also holds in the super-critical case but is not easy to
use to derive properties for super-critical CRT.

 In a second time, we want to construct a ``decreasing''
tree-valued Markov process. To begin with, if $\psi$ is (sub)critical,
for $\theta>0$ we can construct, via the pruning procedure of\vadjust{\goodbreak}
\cite{advplcrt}, from a L\'{e}vy CRT $\ct$ associated with $\psi$,
a~sub-tree $\ct_\theta$ associated with the branching mechanism
$\psi_\theta$ defined by
\[
\forall\lambda\ge0 \qquad
\psi_\theta(\lambda)=\psi(\lambda+\theta)-\psi(\theta).
\]
By \cite{adfalp,vdmfglt}, we can even construct a ``decreasing''
family of L\'{e}vy CRTs $(\ct_\theta,\allowbreak\theta\ge0)$ such that $\ct
_\theta$
is associated with $\psi_\theta$ for every $\theta\ge0$.

In this paper, we consider a critical branching mechanism $\psi$ and
denote by $\Theta$ the set of real numbers $\theta$ (including negative
ones) for which $\psi_\theta$ is a well-defined conservative
branching mechanism (see
Section \ref{sec:propbm} for some examples). Notice that
$\Theta=[\theta_\infty,+\infty)$ or $(\theta_\infty,+\infty)$
for some
$\theta_\infty\in[-\infty,0]$. We then extend the pruning procedure of
\cite{advplcrt} to super-critical branching mechanisms in order to
define a L\'{e}vy CRT-valued process
$(\ct_\theta,\theta\in\Theta)$ such that:
\begin{itemize}
\item for every $\theta\in\Theta$, the L\'{e}vy CRT $\ct_\theta$ is
associated with the branching mechanism $\psi_\theta$;
\item all the trees $\ct_\theta$, $\theta\in\Theta$ have a common root;
\item the tree-valued process $(\ct_\theta,\theta\in\Theta)$ is
decreasing in the sense that for $\theta<\theta'$,
$\ct_{\theta'}$ is a sub-tree of $\ct_\theta$.
\end{itemize}

Let $\rho^\theta$ be the exploration process that codes for $\ct
_\theta$.
We denote by $\bN^\psi$ the excursion measure of the process
$(\rho^\theta,\theta\in\Theta)$, that is under $\bN^\psi$,
each~$\rho^\theta$ is the excursion of an exploration process associated
with $\psi_\theta$. Let $\sigma_\theta$ denote the length of this
excursion. The quantity $\sigma_\theta$ corresponds also to the total
mass of the CB associated with
the tree $\ct_\theta$. We say that the tree $\ct_\theta$ is finite
(under~$\bN^\psi$) if $\sigma_\theta$ is
finite (or equivalently if the total mass of the associated CB is
finite). By construction, we have that the trees $\ct_\theta$ for
$\theta\ge0$ are associated with (sub)critical branching mechanisms
and hence are a.e. finite. On the other hand, the trees $\ct_\theta$
for negative $\theta$ are associated with super-critical branching
mechanisms. We define the explosion time
\[
A=\inf\{\theta\in\Theta, \sigma_\theta<+\infty\}.
\]
For $\theta\in\Theta$, we define $\bar\theta$ as the unique
nonnegative real number such that
%
%
\begin{equation}\label{eq:def-bartheta}
\psi(\bar\theta)=\psi(\theta)
\end{equation}
(notice that $\bar\theta=\theta$ if $\theta\geq0$). If
$\theta_\infty\notin\Theta$, we set $\bar\theta_\infty=\lim
_{\theta\downarrow\theta_\infty}\bar\theta$. We give the
distribution of $A$ under $\bN^\psi$ (Theorem \ref{theo:law_A}). In
particular we have, for all $\theta\in[\theta_\infty,+\infty)$,
\[
\bN^\psi[A>\theta]=\bar\theta-\theta.
\]
We also give the
distribution of the trees after the explosion time $(\ct_\theta,
\theta\geq A)$ (Theorem \ref{theo:loi-sigmaA} and Corollary
\ref{cor:gA}). Of particular interest is the distribution of the
tree at its explosion time, $\ct_A$.

The pruning procedure can been viewed, from a discrete point of view,
as a percolation on a Galton--Watson tree. This idea has been used in
\cite{aptmcdgwp} (percolation on branches) and in \cite{adhpgwttvmp}
(percolation on nodes) to construct tree-valued Markov processes from
a Galton--Watson tree. The CRT-valued Markov process constructed here
can be viewed as the continuous analog of the discrete models of
\cite{aptmcdgwp} and \cite{adhpgwttvmp} (or maybe a mixture of both
constructions). However, no link is actually pointed out between the
discrete and the continuous frameworks.

In \cite{aptmcdgwp} and \cite{adhpgwttvmp}, another representation
of the process up to the explosion time is also given in terms of the
pruning of an infinite tree [a (sub)critical Galton--Watson tree
conditioned on nonextinction].
In the same spirit, we also construct
another tree-valued Markov process $(\ct_\theta^*,\theta\ge0)$
associated with a critical branching mechanism $\psi$. In the case of
a.s. extinction (i.e., when $\int^{+\infty}
\frac{dv}{\psi(v)}<+\infty$), $\ct^*_0$ is distributed as $\ct_0$
conditioned to survival. The tree $\ct^*_0$ is constructed via a spinal
decomposition along an infinite spine. Then we define the
continuum-tree-valued Markov process $(\ct_\theta^*,\theta\ge0)$ again
by a pruning procedure. Let $\theta\in(\theta_\infty,0)$. We prove
that under the excursion measure $\bN^\psi$, given $A=\theta$, the
process $(\ct_{\theta+u}, u\ge0)$ is distributed as the process
$(\ct^*_{\bar\theta+u},u\ge0)$ (Theorem \ref{theo:=loi}).

When the branching mechanism is quadratic, $\psi(\lambda)=\lambda^2/2$,
some explicit computations can be carried out. Let $\sigma^*_\theta$ be
the total mass of $\ct^*_\theta$ and $\tau=(\tau_\theta, \theta
\geq0)$
be the first passage process of a standard Brownian motion, that is a
stable subordinator with index $1/2$. We get (Proposition \ref{prop:st})
that $(\sigma^*_\theta, \theta\geq0)$ is distributed as
$(1/\tau_\theta, \theta\geq0)$ and that $(\sigma_{A+\theta},
\theta\geq
0)$ is distributed as $(1/(V+\tau_\theta), \theta\geq0)$ for some
random variable $V $ independent of $\tau$. Let us recall that the
pruning procedure of the tree can be used to construct some
fragmentation processes (see \cite{aspsf,adfalp,vdmfglt}) and the
process $(\sigma_\theta,\theta\ge0)$, conditionally on $\sigma_0=1$,
represents then the evolution of a tagged fragment. We hence recover a
well-known result of Aldous--Pitman \cite{apsac}: conditionally on
$\sigma_0=1$, $(\sigma_\theta,\theta\ge0)$ is distributed as
$(1/(1+\tau_\theta), \theta\geq0)$ (see Corollary \ref{cor:s1}).

 The paper is organized as follows. In Section
\ref{sec:girsanov_cb}, we introduce an exponential martingale of a CB
and give a Girsanov formula for CBs. We recall in Section \ref{sec:crt}
the construction of a (sub)critical L\'{e}vy CRT via the exploration
process and some useful properties of this exploration process. Then we
construct, in Section \ref{sec:superCRT}, the super-critical L\'{e}vy CRT
via a Girsanov theorem involving the same martingale as in
Section~\ref{sec:girsanov_cb}. We recall in Section~\ref{sec:pruning} the
pruning procedure for critical or sub-critical CRTs and extend this
procedure to super-critical CRTs. We construct in Section
\ref{sec:tree-valued} the tree-valued process
$(\ct_\theta,\theta\in\Theta)$, or more precisely the family of
exploration processes $(\rho^\theta, \theta\in\Theta)$ which
codes for
it. We also give the law of the explosion time~$A$ and the law of the
tree at this time. In Section~\ref{sec:infinite_tree}, we construct an
infinite tree and the corresponding pruned sub-trees
$(\ct_\theta^*,\theta\ge0)$, which are given by a spinal representation
using exploration processes. We prove in Section~\ref{sec:proof_theo}
that the process $(\ct_{A+u},u\ge0)$ is distributed as the process
$(\ct^*_{U+u},u\ge0)$ where $U$ is a positive random time independent
of $(\ct_\theta^*,\theta\ge0)$. We finally make the explicit
computations for the quadratic case in Section~\ref{sec:quadra}.

Notice that all the results in the following sections are stated using
exploration processes which code for the CRT, instead of the CRT
directly. An informal description of the links between the CRT and the
exploration process is given at the end of Section \ref{sec:lien-rho-crt}.


\section{Girsanov's formula for continuous branching process}
\label{sec:girsanov_cb}
\subsection{Continuous branching process}

Let $\psi$ be a branching mechanism of a CB: for $\lambda\geq0$,
%
%
\begin{equation}
\label{eq:def_psi}
\psi(\lambda)=\tilde\alpha\lambda+\beta\lambda^2+ \int
_{(0,{+\infty})}\pi(d\ell)
 \bigl[\mathrm{e}^{-\lambda\ell}-1+\lambda\ell\ind_{\{\ell\leq
1\}
} \bigr],
\end{equation}
where $\tilde\alpha\in\R$, $\beta\geq0$, and $\pi$ is a Radon
measure on
$(0,{+\infty} )$ such that $\int_{(0,{+\infty} )} (1 \wedge\ell
^2)
\pi(d\ell)<{+\infty} $.
We shall say that $\psi$ has parameter $(\tilde\alpha, \beta, \pi)$.

We shall assume that $\beta\neq0$ or $\pi\neq0$. We have $\psi
(0)=0$ and
$\psi'(0^+)=\tilde\alpha- \int_{(1, {+\infty} )} \ell\pi(d\ell)
\in[-\infty
, +\infty)$. In particular, we have $\psi'(0^+)=-\infty$ if and only
if $\int_{(1,{+\infty} )} \ell  \pi(d\ell)={+\infty} $. We say
that $\psi$
is conservative if for all $\varepsilon>0$
%
%
\begin{equation}
\label{eq:conservatif}
\int_0^\varepsilon\frac{1}{| \psi(u) |} \, du ={+\infty} .
\end{equation}
Notice that \reff{eq:conservatif} is fulfilled if $\psi'(0+)>-\infty
, $ that
is, if $\int_{(1,{+\infty} )} \ell  \pi(d\ell)<{+\infty} $.
If $\psi$ is conservative, the CB associated with $\psi$ does not
explode in finite time a.s.

Let $\rP_x^\psi$ be the law of a CB $Z=(Z_a,a\ge0)$ started at
$x\geq
0$ and with branching mechanism $\psi$, and let $\rE_x^\psi$ be the
corresponding expectation. The process $Z$ is a Feller process and thus
has a c\`{a}d-l\`{a}g version. Let \mbox{$\cf=(\cf_a, a\geq0)$} be the
filtration generated by
$Z$ completed the usual way. For every $\lambda>0$, for every $a\ge0$,
we have
%
%
\begin{equation}
\label{eq:laplace_csbp}
\rE^\psi_x [\mathrm{e}^{-\lambda Z_a} ]=\mathrm
{e}^{-xu(a,\lambda)},
\end{equation}
where function $u$ is the unique nonnegative solution of
%
%
\begin{equation}
\label{eq:int_u}
u(a,\lambda)+\int_0^a\psi (u(s,\lambda) )\,ds=\lambda,
\qquad
\lambda\geq0,   a\geq0.
\end{equation}
This equation is equivalent to
%
%
\begin{equation}
\label{eq:int_u2}
\int_{u(a,\lambda)}^\lambda\frac{dr}{\psi(r)}=a, \qquad
\lambda\geq0,   a\geq0.
\end{equation}
If \reff{eq:conservatif} holds, then the process is conservative:
a.s. for all $a\geq0$, $Z_a<+\infty$.

Let $q_0$ be the largest root of $\psi(q)=0$. Since $\psi(0)=0$, we
have $q_0\ge0$. If $\psi$ is (sub)critical, since $\psi$ is strictly
convex, we get that $q_0=0$. If $\psi$ is super-critical, if we denote
by $q^*>0$ the only real number such that $\psi'(q^*)=0$, we have $q_0>q^*>0$.
See Lemma
\ref{lem:propZ} for the interpretation of ${q_0}$.

If $f$ is a function defined on $[\gamma,{+\infty} )$, then for
$\theta\geq
\gamma$, we set for $\lambda\geq\gamma-\theta$
\[
f_\theta(\lambda)=f(\theta+\lambda)-f(\theta).
\]
If $\nu$ is a measure on $(0,{+\infty} )$, then for $q\in\R$,
we set
%
%
\begin{equation}
\label{eq:def-nu-q}
\nu^{(q)}(d\ell)=\mathrm{e}^{-q \ell} \nu(d\ell).
\end{equation}

\begin{rem}
\label{rem:psiq-cons}
If $\pi^{(q)}((1,+\infty))<+\infty$ for some $q<0$, then $\psi$
given by
\reff{eq:def_psi} is well defined on $[q,{+\infty})$ and,
for $\theta\in[q, {+\infty})$, $\psi_\theta$ is a branching
mechanism with parameter $(\tilde\alpha+2\beta\theta+ \int_{(0,1]}
\pi(d\ell)   \ell(1-\mathrm{e}^{-\theta\ell}), \beta, \pi
^{(\theta
)})$. Notice
that for all $\theta>q$, $\psi_\theta$ is conservative. And, if the
additional assumption
\[
\int_{(1,+\infty)} \ell
\pi^{(q)}(d\ell)=\int_{(1,+\infty)}\ell\mathrm{e}^{|q|\ell}\pi
(d\ell
)<+\infty
\]
holds, then $|(\psi_{q})'(0+)|<+\infty$ and $\psi_{q}$ is conservative.
\end{rem}

\subsection{Girsanov's formula}
Let $Z=(Z_a, a\geq0)$ be a conservative CB with branching mechanism
$\psi$ given by \reff{eq:def_psi} with $\beta\neq0$ or $\pi\neq
0$, and
let $(\cf_a, a\geq0)$ be its natural filtration. Let $q\in\R$ such that
$q\geq0$ or $q<0$ and\break
$\int_{(1,+\infty)}\ell\mathrm{e}^{|q|\ell}\pi(d\ell)<+\infty$.
Then, thanks
to Remark \ref{rem:psiq-cons}, $\psi(q)$ and $\psi_q$ are well defined
and $\psi_q$ is conservative. Then we consider the process
$M^{\psi,q}=(M_a^{\psi,q}, a\geq0)$ defined by
%
%
\begin{equation}
\label{eq:defMq}
M_a^{\psi,q}=\mathrm{e}^{qx- q Z_a - \psi(q)\int_0^a Z_s \, ds }.
\end{equation}

\begin{theo}
\label{theo:mart}
Let $q\in\R$ such that $q\geq0$ or $q<0$ and
$\int_{(1,+\infty)}\ell\mathrm{e}^{|q|\ell}\pi(d\ell)<+\infty$.
\begin{longlist}[(ii)]
\item[(i)] The process $M^{\psi,q}$ is a $\cf$-martingale under
$P_x^\psi$.
\item[(ii)] Let $a,x\ge0$. On $\cf_a$, the probability measure
$\rP_x^{\psi_q}$ is absolutely continuous with respect to
$\rP_x^\psi$ and
\[
\frac{{d\rP_x^{\psi_q}}_{|_{\cf_a}}}{{d\rP_x^\psi}_{|_{\cf
_a}}}=M^{\psi,q}_a.
\]
\end{longlist}
\end{theo}

Before going into the proof of this theorem, we recall Proposition 2.1
from~\cite{adcbpimcsbpi}.
For $\mu$ a positive measure on $\R$, we set
%
%
\begin{equation}
\label{def:H}
H(\mu)=\sup\bigl\{r\in\R; \mu\bigl([r,{+\infty} )\bigr)>0\bigr\},
\end{equation}
the maximal element of its support.
For $a<0$, we set $Z_a=0$.
\begin{prop}
\label{prop:uniq_x}
Let $\mu$ be a finite positive measure on $\R$ with support
bounded from above [i.e., $H(\mu)$ is finite]. Then we have for all
$s\in\R$, $x\geq0$,
%
%
\begin{equation}
\label{eq:variante-mu}
\rE_x^\psi \bigl[\mathrm{e}^{-\int_\R Z_{r-s}  \mu(dr) }
\bigr]=\mathrm{e}^{-xw(s)},
\end{equation}
where the function $w$ is a measurable locally bounded nonnegative
solution of the equation
%
%
\begin{eqnarray}
\label{eq:def_w}
w(s)+\int_s^{{+\infty} }\psi(w(r))\,dr&=&\int_{[s,{+\infty} )} \mu
(dr), \qquad s\leq
H(\mu)\quad\mbox{and}\nonumber
\\[-8pt]
\\[-8pt]
 w(s)&=&0, \qquad s> H(\mu).
\nonumber
\end{eqnarray}
If $\psi'(0^+)>-\infty$ or if $\mu(\{H(\mu)\})>0$, then \reff{eq:def_w}
has a unique measurable locally bounded nonnegative
solution.
\end{prop}

\begin{pf*}{Proof of Theorem \ref{theo:mart}}

\textit{First case}. We consider $q>0$ such that $\psi(q)\geq0$.

We have $0\leq M^{\psi,q}_a\leq\mathrm{e}^{qx}$, thus $M^{\psi,q}$ is
bounded. It is clear that $M^{\psi,q}$ is $\cf$-adapted.

To check that $M^{\psi,q}$ is a martingale, thanks
to the Markov property, it is enough to check that $\rE_x^\psi
[M^{\psi,q}_a]=\rE_x^\psi[M^{\psi,q}_0]=1$
for all $a\geq0$ and all $x\ge0$. Consider the measure $\nu_q(dr)=q
\delta_a(dr) +
\psi(q)\ind_{[0,a]}(r)\,dr$, where $\delta_a$ is the Dirac mass at point
$a$. Notice that $H(\nu_q)=a$ and that $\nu_q (\{H(\nu_q)\}
 )=q>0$.
Hence, thanks to Proposition \ref{prop:uniq_x}, there exists a
unique nonnegative solution~$w$ of \reff{eq:def_w} with
$\mu=\nu_q$, and
$\rE_x^\psi[M^{\psi,q}_a]=\mathrm{e}^{-x(w(0)-q)}$. As
$q\ind_{[0,a]}$ also solves~\reff{eq:def_w} with $\mu=\nu_q$, we deduce
that $w=q\ind_{[0,a]}$ and that $\rE_x[M^{\psi,q}_a]=1$. Thus, we
get that $M^{\psi,q}$
is a bounded martingale.

Let $\nu$ be a nonnegative measure on $\R$ with support in
$[0,a]$ [i.e., $H(\nu)\le a$].
Thanks to Proposition \ref{prop:uniq_x}, we have that $
\rE_x^\psi[M^{\psi,q}_a\mathrm{e}^{-\int_\R Z_r \nu(dr)}
]=\mathrm{e}^{-x(v(0)-q)}$, where
$v$ is the unique nonnegative solution of \reff{eq:def_w} with
$\mu=\nu+\nu_q$. As $M^{\psi,q}_a\mathrm{e}^{-\int_\R Z_r \nu(dr)}
\leq M^{\psi,q}_a$, we
deduce that $
\mathrm{e}^{-x(v(0)-q)}=\rE_x^\psi[M^{\psi,q}_a\mathrm{e}^{-\int
_\R Z_r \nu (dr)} ]\leq1$,
that is, $v(0)\geq q$. We set $u=v-q\ind_{[0,a]}$, and we deduce that
$u$ is
nonnegative and solves
%
%
\begin{eqnarray}\label{eq:def_u}
u(s)+\int_s^{+\infty} \psi_q(u(r))\, dr&=&\int_{[s,{+\infty} )} \nu(dr),
\qquad s\leq H(\nu) \quad\mbox{and}\nonumber
\\[-8pt]
\\[-8pt]
 u(s)&=&0, \qquad s> H(\nu).
\nonumber
\end{eqnarray}
As $\psi(q)\geq0$, we deduce from the convexity of $\psi$ that $\psi
_q'(0)=\psi'(q)\geq0$. Thanks to
Proposition \ref{prop:uniq_x}, we deduce that $u$ is the
unique nonnegative solution of \reff{eq:def_u} and that
$
\mathrm{e}^{-xu(0)}=\rE_x^{\psi_q}[\mathrm{e}^{- \int_\R Z_r \nu
(dr)}]$. In
particular,
we have that for all nonnegative measure $\nu$ on $\R$ with support in
$[0,a]$,
\[
\rE^ \psi_x\bigl [M^{\psi,q}_a\mathrm{e}^{-\int_\R Z_r \nu(dr)}
\bigr]=\rE_x^{\psi_q} \bigl[\mathrm{e}^{- \int_\R Z_r \nu(dr)} \bigr].
\]

As $\mathrm{e}^{-\int_\R Z_r \nu(dr)}$ is $\cf_a$-measurable, we
deduce from
the monotone class theorem that for any nonnegative $\cf
_a$-measurable random variable $W$,
%
%
\begin{equation}
\label{eq:EMq}
\rE^\psi_x \bigl[W \mathrm{e}^{q x -q Z_a- \psi(q)\int_0^a Z_r\,
dr} \bigr]
=\rE^{\psi}_x[WM^{\psi,q}_a ]
=\rE_x^{\psi_q}[W].
\end{equation}
This proves the second part of the theorem.

\textit{Second case}. We consider $q\geq0$ such that $\psi(q)<
0$. Let us remark that this only occurs when $\psi$ is super-critical.

Recall that $q_0>q^*>0$ are such that $\psi(q_0)=0$ and $\psi'(q^*)=0$.
Notice that $\psi_{{q^*}}'(0)=\psi'({q^*})=0$, that is, $\psi
_{{q^*}}$ is
critical. Let $W$ be any nonnegative random variable
$\cf_a$-measurable. From the first step, using \reff{eq:EMq} with
$q={q_0}$, we get that
\[
\rE^\psi_x [W \mathrm{e}^{{q_0} x -{q_0} Z_a} ]=\rE
_x^{\psi_{{q_0}}}[W].
\]
Thanks to
\reff{eq:EMq} with
$\psi_{{q^*}}$ instead of $\psi$ and $({q_0}-{q^*})\geq0$ instead of
$q$, and
using that $ (\psi_{{q^*}} )_{{q_0}-{q^*}}=\psi_{{q_0}}$,
we deduce
that
\[
\rE^{\psi_{{q^*}}}_x\bigl [W \mathrm{e}^{({q_0}-{q^*})x -
({q_0}-{q^*})Z_a - \psi_{{q^*}}({q_0}-{q^*}) \int_0^a Z_r \, dr}
 \bigr] =\rE_x^{
(\psi_{{q^*}} )_{{q_0}-{q^*}}
}[W]=\rE_x^{\psi_{{q_0}}}[W].
\]
This implies that
\begin{eqnarray*}
\rE_x^{\psi}[W] &
=&\rE_x^{\psi_{{q_0}}} [W\mathrm{e}^{-{q_0}x}\mathrm
{e}^{{q_0}Z_a} ]\\
& =&\rE_x^{\psi_{{q^*}}} \bigl[ W\mathrm{e}^{-{q_0}x}\mathrm
{e}^{{q_0}Z_a}\mathrm{e}^{({q_0}-{q^*})x - ({q_0}-{q^*})Z_a - \psi
_{{q^*}}({q_0}-{q^*}) \int_0^a Z_r \, dr}  \bigr]\\
& =&\rE_x^{\psi_{{q^*}}} \bigl[W \mathrm{e}^{-{q^*}x +{q^*} Z_a -
\psi_{{q^*}}({q_0}-{q^*}) \int_0^a Z_r \, dr}  \bigr].
\end{eqnarray*}
As
$\psi_{{q^*}}({q_0}-{q^*})=\psi({q_0})-\psi({q^*})=-\psi
({q^*})=\psi_{{q^*}}(-{q^*})$,
we finally obtain
%
%
\begin{equation}
\label{eq:EMq2}
\rE_x^{\psi}[W ]=\rE_x^{\psi_{{q^*}}} \bigl[W
\mathrm{e}^{-{q^*}x +{q^*} Z_a -\psi_{{q^*}}(-{q^*}) \int_0^a Z_r \,
dr}  \bigr].
\end{equation}

If $q<{q^*}$, as $(\psi_q)_{({q^*}-q)}=\psi_{{q^*}}$ and
$\psi'_q({q^*}-q)=\psi'({q^*})=0$, we deduce
from~\reff{eq:EMq2} with $\psi$ replaced by $\psi_q$ and ${q^*}$ by ${q^*}-q$
that
%
%
\begin{equation}
\label{eq:EMq2bis}
\rE_x^{\psi_q}[W ]=\rE_x^{\psi_{{q^*}}} \bigl[W
\mathrm{e}^{-({q^*}-q)x +({q^*}-q) Z_a -\psi_{{q^*}}(q-{q^*}) \int
_0^a Z_r \, dr}  \bigr].
\end{equation}
If $q>{q^*}$, formula \reff{eq:EMq}
holds with $\psi$ replaced by $\psi_{{q^*}}$ and $q$ replaced by $q-{q^*}$,
which also yields equation \reff{eq:EMq2bis}.

Using \reff{eq:EMq2}, \reff{eq:EMq2bis} and that
$\psi_{{q^*}}(-{q^*})+\psi(q)=\psi_{{q^*}}(q-{q^*})$, we get that
%
\begin{eqnarray}
\label{eq:EMq3}
&&\rE^\psi_x \bigl[W \mathrm{e}^{q x -q Z_a- \psi(q)\int_0^a Z_r\,
dr} \bigr]\nonumber\\
&& \qquad = \rE_x^{\psi_{{q^*}}} \bigl[W
\mathrm{e}^{-({q^*}-q)x +({q^*}-q) Z_a -(\psi_{{q^*}}(-{q^*})+\psi
(q)) \int_0^a Z_r \, dr}  \bigr]\nonumber
\\[-8pt]
\\[-8pt]
&& \qquad = \rE_x^{\psi_{{q^*}}} \bigl[W
\mathrm{e}^{-({q^*}-q)x +({q^*}-q) Z_a -\psi_{{q^*}}(q-{q^*}) \int
_0^a Z_r \, dr}  \bigr]\nonumber\\
&& \qquad = \rE_x^{\psi_q}[W ].
\nonumber
\end{eqnarray}
Since this holds for any nonnegative $\cf_a$-measurable random variable
$W$, this proves
(i) and (ii) of the theorem.\vadjust{\goodbreak}

\textit{Third case}. We consider $q<0$ and assume that $\int
_{(1,{+\infty}
)} \ell\mathrm{e}^{|q|\ell}   \pi(d \ell)<{+\infty} $. In particular,
$\psi_q$ is a conservative branching mechanism, thanks to Remark~\ref{rem:psiq-cons}.

Let $W$ be
any nonnegative $\cf_a$-measurable random variable. Using \reff{eq:EMq}
if $\psi_q (-q)\geq0$ or \reff{eq:EMq3} if $\psi_q
(-q)< 0$, with $\psi$ replaced by $\psi_q$ and $q$ by $-q$,
we deduce that
\[
\rE^{\psi_q}_x\bigl[W \mathrm{e}^{-q x +q Z_a- \psi_q(-q)\int_0^a Z_r\, dr}\bigr]
=\rE_x^{\psi}[W].
\]
This implies that
\[
\rE^{\psi_q}_x[W ]
=\rE_x^{\psi}\bigl[W\mathrm{e}^{q x -q Z_a+ \psi_q(-q)\int_0^a Z_r\, dr}\bigr]
=\rE_x^{\psi}\bigl[W\mathrm{e}^{q x -q Z_a- \psi(q)\int_0^a Z_r\, dr}\bigr].
\]
Since this holds for any nonnegative $\cf_a$-measurable random variable
$W$, this proves (i) and (ii) of the theorem.
\end{pf*}

Finally, we recall some well-known facts on CB. Recall that $q_0$ is
the largest root of $\psi(q)=0$, $q_0=0$ if $\psi$ is (sub)critical and
that $q_0>0$ if $\psi$ is super-critical. We set
%
%
\begin{equation}
\label{eq:def-sigma}
\sigma=\int_0^{+\infty} Z_a\, da.
\end{equation}
For $\lambda\geq0$, we set
%
%
\begin{equation}
\label{eq:def-psi-1}
\psi^{-1}(\lambda)=\sup\{r\geq0;   \psi(r)=\lambda\},
\end{equation}
and we call $\sigma$ the total mass of the CB.

\begin{lem}
\label{lem:propZ}
Assume that $\psi$ is given by \reff{eq:def_psi} with $\beta\neq0$ or
$\pi\neq0$ and is conservative.
\begin{longlist}[(iii)]
\item[(i)] Then $\rP^
\psi_x$-a.s. $Z_\infty=\lim_{a\rightarrow+\infty}Z_a$ exists,
$Z_\infty\in\{0, +\infty\}$,
%
%
\begin{equation}
\label{eq:PZI=0}
\rP^\psi_x(Z_\infty=0
)=\mathrm{e}^{-xq_0},
\end{equation}
$\{Z_\infty=0\}=\{\sigma<+\infty\}$, and we have, for
$\lambda>0$,
%
%
\begin{equation}
\label{eq:IZinfini}
\rE_x^\psi [\mathrm{e}^{- \lambda\sigma}
 ]=\mathrm{e}^{-x \psi^{-1}(\lambda)}.
\end{equation}
\item[(ii)] Let $q>0$ such that $\psi(q)\geq0$. Then, the probability
measure $\rP_x^{\psi_q}$ is absolutely continuous with respect to
$\rP_x^\psi$ with
\[
\frac{d\rP_x^{\psi_q}}{d\rP_x^\psi}=M^{\psi,q}_\infty,
\]
where
%
%
\begin{equation}
\label{eq:Minfini}
M^{\psi,q}_\infty=\mathrm{e}^{qx - \psi(q) \sigma}\ind_{\{\sigma
<+\infty\}}.
\end{equation}
\item[(iii)] If $\psi$ is super-critical then, conditionally on
$\{Z_\infty=0\}$, $Z$ is distri\-buted as $\rP^{\psi_{{q_0}}}$: for any
nonnegative
random variable measurable w.r.t. $\sigma(Z_a,\allowbreak a\geq0)$, we have
\[
\rE^\psi_x [W|Z_\infty=0 ]=\rE^{\psi_{{q_0}}}_x[W].\vadjust{\goodbreak}
\]
\end{longlist}
\end{lem}

\begin{pf}
For $\lambda>0$, we set $N_a= \mathrm{e}^{-\lambda Z_a + x
u(a,\lambda)}$,
where $u $ is the unique nonnegative solution of \reff{eq:int_u2}.
Thanks to \reff{eq:laplace_csbp} and the Markov property, $ (N_a,
a\geq0 )$ is a bounded martingale under $\rP^\psi_x$. Hence,
as $a$ goes to infinity, it converges a.s. and in $L^1$ to a limit, say
$N_\infty$. From \reff{eq:int_u2}, we get that $\lim_{a\rightarrow
+\infty} u(a,\lambda)=q_0$. This implies that $Z_\infty
=\lim_{a\rightarrow+\infty} Z_a$ exists a.s. in $[0, +\infty
]$. Since $\rE^\psi_x [N_\infty]=1$, we get $\rE^\psi_x [\mathrm
{e}^{-\lambda Z_\infty} ]=\mathrm{e}^{-q_0 x}$ for all $\lambda>0$.
This implies that
$\rP^\psi_x$-a.s. $Z_\infty\in\{0, +\infty\}$ and \reff{eq:PZI=0}.

Clearly, we have $\{Z_\infty=+\infty
\}\subset\{\sigma=+\infty\}$. For $q> 0$ such that
$\psi(q)\geq0$, we get that $(M^{\psi, q}_a, a\geq0)$ is a bounded
martingale under $\rP^\psi_x$. Hence, as $a$ goes to infinity, it
converges a.s. and in $L^1$ to a limit, say $M_\infty^{\psi,q} $. We
deduce that
%
%
\begin{equation}
\label{eq:Zinfini}
\rE^\psi_x  \bigl[\mathrm{e}^{- \psi(q) \sigma}\ind_
{\{Z_\infty=0\}}  \bigr] =\mathrm{e}^{-qx}.
\end{equation}
Letting $q$ decrease to $q_0$, we get that $\rP^{\psi}_x
(\sigma
<+\infty, Z_\infty=0  )=\mathrm{e}^{-q_0 x} =\break\rP_x^\psi
(Z_\infty
=0 )$. This
implies that $\rP^\psi_x$ a.s. $\{\sigma=+\infty\}\subset\{
Z_\infty
=+\infty\}$. We thus deduce that $\rP^\psi_x$ a.s. $\{Z_\infty
=+\infty\}= \{\sigma=+\infty\}$. Notice also that \reff{eq:Minfini}
holds.

Notice that \reff{eq:Zinfini} readily implies \reff{eq:IZinfini}.
This proves Property (i) of the lemma and \reff{eq:Minfini}.

Property (ii) is then a consequence of Theorem \ref{theo:mart},
Property (ii) and
the convergence in $L^1$ of the martingale $(M^{\psi, q}_a, a\geq0)$
towards $M_\infty^{\psi,q} $.

Property (iii) is a consequence of (ii) with $q=q_0$ and \reff{eq:PZI=0}.
\end{pf}


\section{L\'{e}vy continuum random tree}\label{sec:crt}

We recall here the construction of the L\'{e}vy continuum random tree (CRT)
introduced in \cite{lgljbplpep,lgljbplplfss} and developed later in
\cite{dlgrtlpsbp} for critical or sub-critical branching mechanism. We
will emphasize on the height process and the exploration process which
are the key tools to handle this tree. The results of this section are
mainly extracted from \cite{dlgrtlpsbp}, except for the next
subsection which is extracted from \cite{lgrrt}.

\subsection{Real trees and their coding by a continuous function}

Let us first define what a real tree is.

\begin{defi}
A metric space $(\ct,d)$ is a real tree if the following two
properties hold for every $v_1,v_2\in\ct$:
\begin{longlist}[(ii)]
\item[(i)]  {(unique geodesic)}
There is a unique isometric map $f_{v_1,v_2}$
from $[0,d(v_1,\allowbreak v_2)]$ into $\ct$ such that
\[
f_{v_1,v_2}(0)=v_1\quad\mbox{and}\quad
f_{v_1,v_2}(d(v_1,v_2))=v_2.
\]
\item[(ii)]  {(no loop)}
If $q$ is a continuous injective map from $[0,1]$ into
$\ct$ such that $q(0)=v_1$ and $q(1)=v_2$, we have
\[
q([0,1])=f_{v_1,v_2}([0,d(v_1,v_2)]).
\]
\end{longlist}
A rooted real tree is a real tree $(\ct,d)$ with a distinguished
vertex $v_\varnothing$ called the root.
\end{defi}

Let $(\ct,d)$ be a rooted real tree. The range of the mapping
$f_{v_1,v_2}$ is denoted by $\lb v_1,v_2,\rb$ (this is the line
between $v_1$ and $v_2$ in the tree). In particular, for every
vertex $v\in\ct$, $\lb v_\varnothing,v\rb$ is the path going from the
root to $v$ which we call the ancestral line of vertex $v$. More
generally, we say that a vertex $v$ is an ancestor of a vertex $v'$
if $v\in\lb v_\varnothing,v'\rb$. If $v,v'\in\ct$, there is a unique
$a\in\ct$ such that $\lb
v_\varnothing,v\rb\cap\lb v_\varnothing,v'\rb=\lb v_\varnothing, a\rb
$. We
call $a$ the most recent common ancestor to $v$ and $v'$. By
definition, the degree of a vertex $v\in\ct$ is the number of
connected components of $\ct\setminus\{v\}$. A vertex $v$ is called
a leaf if it has degree~1. Finally, we set $\lambda$ the
one-dimensional Hausdorff measure on~$\ct$.

The coding of a compact real tree by a continuous function is
now well known and is a key tool for defining random real trees (see Figure~\ref{fig1}). We
consider a continuous function $g \dvtx   [0,+\infty)\longrightarrow
[0,+\infty)$ with compact support and such that $g(0)=0$. We also
assume that $g$ is not identically 0. For every $0\le s\le t$, we set
\[
m_g(s,t)=\inf_{u\in[s,t]}g(u)
\]
and
\[
d_g(s,t)=g(s)+g(t)-2m_g(s,t).
\]
We then introduce the equivalence relation $s\sim t$ if and only if
\mbox{$d_g(s,t)=0$}. Let $\ct_g$ be the quotient space $[0,+\infty)/\sim$. It
is easy to check that $d_g$ induces a distance on $\ct_g$. Moreover,
$(\ct_g,d_g)$ is a compact real tree (see \cite{dlgpfalt}, Theorem~2.1).
We say that $g$ is the height process of the tree $\ct_g$.

%
\begin{figure}

\includegraphics{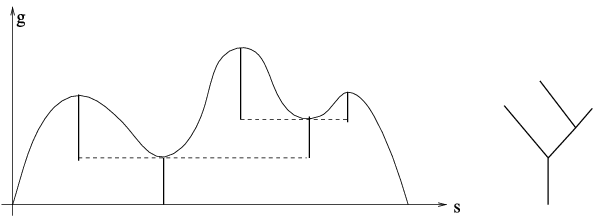}

\caption{A height process $g$ and its associated real
tree.}\label{fig1}
\end{figure}

In order to define a random tree, instead of taking a tree-valued
random variable (which implies defining a $\sigma$-field on the set of
real trees), it suffices to take a continuous stochastic process for
$g$. For instance,
when $g$ is a~normalized Brownian excursion, the associated real tree is
Aldous's CRT (up to a factor 2) \cite{acrt3}. We present now how we
can define a height
process that codes a random real trees describing the genealogy of a
(sub)critical CB with branching mechanism $\psi$. This height process
is defined via a L\'{e}vy process that we first introduce.

\subsection{The underlying L\'{e}vy process}\label{subsec:levy}
We assume that $\psi$ given by \reff{eq:def_psi} is (sub)criti\-cal,
that is,
%
%
\begin{equation}\label{eq:def_alpha}
\alpha:=\psi'(0)=\tilde\alpha-\int_{(1,+\infty)} \ell
\pi(d\ell)\geq0
\end{equation}
and that
%
%
\begin{equation}\label{eq:infinite_variation}
\beta>0\quad\mbox{or}\quad\int_{(0,1)}\ell\pi(d\ell)=+\infty.
\end{equation}

We consider a $\R$-valued L\'{e}vy process $X=(X_t,t\ge
0)$ with no negative jumps, starting from 0 and with Laplace exponent
$\psi$ under the probability measure $\P^\psi$: for $\lambda\ge0$
$\E^\psi [\mathrm{e}^{-\lambda X_t}
]=\mathrm{e}^{t\psi(\lambda)}$.
By assumption \reff{eq:infinite_variation}, $X$ is of infinite
variation $\P^\psi$-a.s.

We introduce some processes related to $X$. Let $\cj=\{s\geq0;
X_s\neq
X_{s-}\}$ be the set of jump times of $X$. For $s\in\cj$, we denote by
\[
\Delta_s= X_s- X_{s-}
\]
the size of the jump of $X$ at time $s$ and $\Delta_s=0$ otherwise.
Let $I=(I_t,t\ge0)$
be the infimum process of $X$,
\[
I_t=\inf_{0\le s\le t}X_s,
\]
and let
$S=(S_t,t\ge0)$ be the supremum process,
\[
S_t=\sup_{0\le s\le t}X_s.
\]
We will also consider for every $0\le s\le t$ the infimum of $X$ over
$[s,t]$,
\[
I_t^s=\inf_{s\le r\le t}X_r.
\]

The point 0 is regular for the Markov process $X-I$, and $-I$ is the
local time of $X-I$ at 0 (see \cite{bpl}, Chapter VII). Let $\N^\psi
$ be
the associated excursion measure of the process $X-I$ away from 0. Let
$\sigma=\inf\{t>0; X_t-I_t=0\}$ be the length of the excursion of $X-I$
under $\N^\psi$ [we shall see after Proposition~\ref{prop:RK} that the
notation $\sigma$ is consistent with~\reff{eq:def-sigma}]. By
assumption~\reff{eq:infinite_variation}, we have $X_0=I_0=0$ $\N^\psi$-a.e.

Since $X$ is of infinite variation, 0 is also regular for the Markov
process $S-X$. The local time, $L=(L_t, t\geq0)$, of $S-X$ at 0 will be
normalized so that
\[
\E^\psi[\mathrm{e}^{-\lambda S_{L^{-1}_t}}]= \mathrm{e}^{- t \psi
(\lambda )/\lambda},
\]
where $L^{-1}_t=\inf\{ s\geq0; L_s\geq t\}$ (see also \cite{bpl}
Theorem VII.4(ii)).

\subsection{The height process and the L\'{e}vy CRT}
For each $t\geq0$, we consider the reversed process at time $t$,
$\hat X^{(t)}=(\hat X^{(t)}_s,0\le s\le t)$ by
\[
\hat X^{(t)}_s=
X_t-X_{(t-s)-} \qquad\mbox{if } 0\le s<t,
\]
and $\hat X^{(t)}_t=X_t$. The two processes $(\hat X^{(t)}_s,0\le s\le t)$
and $(X_s,0\le s\le t)$ have the same law. Let $\hat S^{(t)}$ be the
supremum process of $\hat X^{(t)}$ and $\hat L^ {(t)}$ be the
local time at $0$ of $\hat S^{(t)} - \hat X^{(t)}$ with the same
normalization as $L$.

\begin{defi}[(\cite{dlgrtlpsbp}, Definition
1.2.1)]\label{def:height_process}
There exists a lower semi-continu\-ous modification of the process $(\hat
L^{(t)},t\ge0)$. We denote by $(H_t,t\ge0)$ this modification.
\end{defi}

We can also define this process $H$ by approximation: it is a
modification of the process
%
%
\begin{equation}\label{eq:defHbis}
H_t^0=\liminf_{\varepsilon\to
0}\frac{1}{\varepsilon}\int_0^t\ind_{\{X_s<I_t^s+\varepsilon\}}\,ds
\end{equation}
(see \cite{dlgrtlpsbp}, Lemma 1.1.3). In general,
$H$ takes its values in $[0,+\infty]$, but we have that, a.s. for every
$t\ge0$:
\begin{itemize}
\item$H_s<+\infty$ for every $s<t$ such that $X_{s-}\le I_t^s$;
\item$H_t<+\infty$ if $\Delta X_t>0$
\end{itemize}
(see \cite{dlgrtlpsbp}, Lemma
1.2.1).

We use this process to define a random real-tree that we call the
$\psi$-L\'{e}vy CRT via the procedure described above. We will see that
this CRT does represent the genealogy of a $\psi$-CB.

\subsection{The exploration process}
\label{sec:PLRT}
The height process is not Markov in general. But it is a very simple
function of a measure-valued Markov process, the so-called exploration
process.

If $E$ is a locally compact polish space, let $\cb(E)$ [resp., $\cb
_+(E)$] be the set of
real-valued measurable (resp., and nonnegative) functions defined on $E$
endowed with its Borel $\sigma$-field, and let $\cm(E)$
[resp., $\cm_f(E)$] be the set of $\sigma$-finite (resp., finite)
measures on $E$, endowed with the topology of vague (resp., weak)
convergence. For any measure $\mu\in\cm(E)$ and $f\in\cb_+(E)$, we write
\[
\langle\mu,f\rangle=\int_{E} f(x) \mu(dx).
\]

The exploration process $\rho=(\rho_t,t\ge0)$ is a
$\cm_f(\R_+)$-valued process defined as follows: for every $f\in
\cb_+(\R_+) $, $\langle\rho_t,f\rangle=\int_{[0,t]} d_sI_t^sf(H_s)$
(where $d_sI_t^s$ denotes the Lebesgue--Stieljes\vadjust{\goodbreak} integral with respect
to the nondecreasing map $s\mapsto I_t^s$),
or equivalently
%
%
\begin{equation}\label{eq:def_rho}
\rho_t(dr)=\mathop{\mathop{\sum}_{{0<s\le t}}}_
{X_{s-}<I_t^s}(I_t^s-X_{s-})\delta_{H_s}(dr)+\beta\ind_{[0,H_t]}(r)\,dr.
\end{equation}
In particular, the total mass of $\rho_t$ is
$\langle\rho_t,1\rangle=X_t-I_t$.

Recall the definition \reff{def:H} of $H(\mu)$ for a measure $\mu$
with compact support and set by
convention $H(0)=0$.

\begin{prop}[(\cite{dlgrtlpsbp}, Lemma 1.2.2 and
formula (1.12))]
\label{prop:rho}
Almost surely, for every $t>0$:
\begin{itemize}
\item$H(\rho_t)=H_t$;
\item$\rho_t=0$ if and only if $H_t=0$;
\item if $\rho_t\neq0$, then $\Supp\rho_t=[0,H_t]$;
\item$\rho_t= \rho_{t^-} + \Delta_t \delta_{H_t}$, where $\Delta_t=0$
if $t\notin\cj$.
\end{itemize}
\end{prop}

In the definition of the exploration process, as $X$ starts from 0, we
have $\rho_0=0$ a.s. To state the Markov property of $\rho$, we must
first define the process~$\rho$ started at any initial measure $\mu
\in
\cm_f(\R_+)$.

For $a\in[0, \langle\mu,1\rangle] $, we define the erased measure
$k_a\mu$ by
\[
k_a\mu([0,r])=\mu([0,r])\wedge(\langle\mu,1\rangle-a)  \qquad
\mbox{for $r\geq0$}.
\]
If $a> \langle\mu,1\rangle$, we set $k_a\mu=0$. In other words, the
measure $k_a\mu$ is the measure~$\mu$ erased by a mass $a$ backward from
$H(\mu)$.

For $\nu,\mu\in\cm_f(\R_+)$, and $\mu$ with compact support, we
define the concatenation $[\mu,\nu]\in\cm_f(\R_+) $ of the
two measures by
\[
 \langle[\mu,\nu],f \rangle= \langle\mu,f
\rangle+ \bigl\langle\nu,f\bigl(H(\mu)+\cdot\bigr) \bigr\rangle,
\qquad f\in\cb_+(\R_+).
\]

Finally, we set for every $\mu\in\cm_f(\R_+)$ and every $t>0$,
$\rho_t^\mu= [k_{-I_t}\mu,\rho_t]$. We say that $(\rho^\mu
_t, t\geq
0)$ is the process $\rho$ started at $\rho_0^\mu=\mu$. Unless there is
an ambiguity, we shall write $\rho_t$ for $\rho^\mu_t$. Unless it is
stated otherwise, we assume that $\rho$ is started at $0$.

\begin{prop}[(\cite{dlgrtlpsbp}, Proposition 1.2.3)]
The process $(\rho_t,t\ge0)$ is a~c\`{a}d-l\`{a}g strong Markov
process in
$\cm_f(\R_+)$.
\end{prop}

\begin{rem}
\label{rm:rho-L}
{F}rom the construction of $\rho$, we get that a.s.
$\rho_t=0$ if and only if $ -I_t\geq{\langle\rho_0,1\rangle}$ and
$X_t-I_t=0$. This implies that $0$ is also a regular point for
$\rho$.
Notice that $\N^\psi$ is also the excursion measure of the process
$\rho$
away from $0$, and that $\sigma$, the length of the excursion, is
$\N^\psi$-a.e. equal to $\inf\{ t>0; \rho_t=0\}$.
\end{rem}

\subsection{Notations}

We consider the set $\cd$ of c\`{a}d-l\`{a}g processes in $\cm_f(\R_+)$,
endowed with the Skorohod topology and the Borel $\sigma$-field. In what
follows, we denote by $\rho=(\rho_t,t\ge0)$ the canonical process
on\vadjust{\goodbreak}
this set. We still denote by $\P^\psi$ the probability measure on
$\cd$
such that the canonical process is distributed as the
exploration process associated with the branching mechanism $\psi$, and
by $\N^\psi$ the corresponding excursion measure.

\subsection{Local time of the height process}\label{sec:lien-rho-crt}
The local time of the height process is defined through the next
result.
\begin{prop}[(\cite{dlgrtlpsbp}, Lemma 1.3.2 and Proposition 1.3.3)]
\label{prop:LT}
There exists a jointly measurable process $(L^a_s, a\geq0, s\geq0)$
which is continuous and nondecreasing in the variable $s$ such that:
\begin{itemize}
\item
for every $t\geq0$, $\lim_{\varepsilon
\rightarrow0} \sup_{a\geq0} \E^\psi [\sup_{s\leq t}
 |\varepsilon^{-1} \int_0^s \ind_{\{ a<H_r\leq a+\varepsilon\}
}\,
dr - L^a_s | ]=0$;
\item
for every $t\geq0$, $\lim_{\varepsilon
\rightarrow0} \sup_{a\geq\varepsilon} \E^\psi [\sup_{s\leq t}
 |\varepsilon^{-1} \int_0^s \ind_{\{ a-\varepsilon<H_r\leq a\}
}\,
dr - L^a_s | ]=0$;
\item$\P^\psi$-a.s., for every $t\geq0$, $L^0_t=-I_t$;
\item the occupation time formula holds: for any nonnegative
measurable function $g$ on $\R_+$ and any
$s\geq0$, $\int_0^s g(H_r)\, dr =\int_{(0,{+\infty} )}
g(a) L^a_s\, da$.
\end{itemize}
\end{prop}

Let $T_x=\inf\{t\geq0; I_t\leq-x\}$. We have the following Ray--Knight
theorem which links the $\psi$-L\'{e}vy CRT with the $\psi$-CB.
\begin{prop}[(\cite{dlgrtlpsbp}, Theorem 1.4.1)]
\label{prop:RK}
The process $(L^a_{T_{x}}, a\geq0)$ is distributed under $\P^\psi$
as $Z$ under
$\rP_x^\psi$ (i.e., is a CB with branching mechanism~$\psi$ starting
at $x$).
\end{prop}

Let
$\P^{\psi}_x$ be the distribution of $(\rho_{t\wedge T_x},
t\geq0)$ under $\P^{\psi}$.
We set $Z_a=L^a_{T_x}$ under $\P^\psi_x$ and $Z_a=L^a_\infty$ under
$\N^\psi$ and (under $\P^\psi_x$ or $\N^\psi$)
%
%
\begin{equation}
\label{eq:def-s}
\sigma(\rho)=\int_0^\infty\ind_{\{\rho_t\neq0\}}\, dt.
\end{equation}
The occupation time formula implies that
%
%
\begin{equation}\label{eq:defs2}
\sigma(\rho)=\int_0^{+\infty}
Z_a \,da,
\end{equation}
which is consistent with notation \reff{eq:def-sigma}.
When there is no confusion, we shall write $\sigma$ for
$\sigma(\rho)$. We call $\sigma(\rho)$ the total mass of the CRT as it
represents the total population of the associated CB.

Exponential formula for the Poisson point process of jumps of the
inverse subordinator of $-I$ gives (see also the beginning of Section
3.2.2. \cite{dlgrtlpsbp}) that for $\lambda>0$
%
%
\begin{equation}
\label{eq:N_s}
\N^\psi [1 -\mathrm{e}^{-\lambda \sigma} ] =\psi
^{-1}(\lambda).
\end{equation}

We also recall Lemma 1.6 of \cite{adfalp}.
\begin{lem}
\label{lem:X_pruned}
Let
$\theta>0$. The excursion measure $\N^{\psi_\theta}$ is absolutely
continuous w.r.t.\vadjust{\goodbreak} $\N^{\psi}$ with density $\mathrm{e}^{-\psi
(\theta) \sigma}$: for any nonnegative measurable function~$F$
on the space of excursions, we have
\[
\N^{\psi_\theta}  [ F(\rho) ]
=\N^{\psi} \bigl[ F(\rho)\mathrm{e}^{-\psi(\theta)\sigma}
\bigr].
\]
\end{lem}

We recall the Poisson representation of $\P^\psi_x$ based on the
excursion measure~$\N^\psi$. Let
$(\tilde\alpha_i,\tilde\beta_i)_{i\in\tilde I}$ be the excursion intervals
of $\rho$ away from $0$. For every $i\in\tilde I$, $t\ge0$, we set
\[
\tilde\rho^{(i)}_t=\rho_{(\tilde\alpha_i+t)\wedge\tilde\beta_i}.
\]
We
deduce from Lemma 4.2.4 of \cite{dlgrtlpsbp} the following lemma.

\begin{lem}\label{lem:poisson_repr}
The point measure $\sum_{i\in\tilde
I}\delta_{\tilde\rho^{(i)}}(d\mu)$ is under $\P_x^\psi$ a Poisson
measure with intensity $x\N^\psi(d\mu)$.
\end{lem}

To better understand the links between the L\'{e}vy CRT and the
exploration process, we can combine the Markov property with the other
Poisson decomposition of \cite{dlgrtlpsbp}, Lemma 4.2.4. Informally
speaking, the measure $\rho_t$ is a measure placed on the ancestral
line of the individual labelled $t$ which describes how the sub-trees
``on the right'' of $t$ (i.e., containing individuals $s\ge t$) are
grafted along that ancestral line. More precisely, if we
denote~$(\ct_i)_{i\in\ci}$ the family of these subtrees and we set $h_i$ the
height where the subtree $\ct_i$ branches from the ancestral line of
$t$, then the family $(h_i,\ci_i)_{i\in\ci}$ given $\rho_t$ is
distributed as the atoms of a Poisson measure with intensity
$\rho_t(dh)\N^\psi[d\ct]$ (see Figure \ref{fig:descrip-rho}).

%
\begin{figure}

\includegraphics{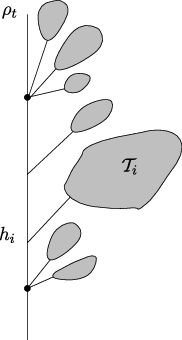}

\caption{The measure $\rho_t$ and the family $(h_i,\ci_i)_{i\in\ci
}$.}\label{fig:descrip-rho}
\end{figure}

As the measure $\N^\psi$ is an infinite measure, we see that the
branching points along the ancestral line of $t$ are of two types (see
\cite{dlgpfalt}, Theorem 4.6):
\begin{itemize}
\item binary nodes (i.e., vertex of degree 3) which are given by the
regular part of $\rho_t$,
\item infinite nodes (i.e., vertex of infinite degree) which are given
by the atomic part of $\rho_t$.
\end{itemize}
By the definition of $\rho_t$, we see that these infinite nodes are
associated with the jumps of the L\'{e}vy process $X$. If such a node
corresponds to a jump time~$s$ of $X$, we call $\Delta X_s$ the size
of the node.

\subsection{The dual process and representation formula}
\label{sec:dual}

We shall need the  $\cm_f(\R_+)$-valued process $\eta=(\eta_t,t\ge0)$
defined by
%
%
\begin{equation}
\label{eq:eta}
\eta_t(dr)=\mathop{\mathop{\sum}_{{0<s\le t}}}_{X_{s-}<I_t^s}(X_s-I_t^s)\delta
_{H_s}(dr)+\beta\ind_{[0,H_t]}(r)\,dr.
\end{equation}
The process $\eta$ is the dual process of $\rho$ under $\N^\psi$ (see
Corollary 3.1.6 in \cite{dlgrtlpsbp}). It represents how the trees
``on the left'' of $t$ branch along the ancestral line of $t$.


We recall the Poisson representation of $(\rho,\eta)$ under $\N^\psi
$. Let
$\mathcal{N}(dx\, d\ell\, du)$ be a Poisson point measure on
$[0,+\infty)^3$ with intensity
\[
dx \ell\pi(d\ell)\ind_{[0,1]}(u)\,du.
\]
For every $a>0$, let us denote by $\mathbb{M}_a^\psi$ the law of the pair
$(\mu_a,\nu_a)$ of measures on $\R_+$ with finite mass defined by
the following: for any $f\in\cb_+(\R_+)$
%
\begin{eqnarray}
\label{def:mu_a}
\langle\mu_a,f\rangle
&=& \int\mathcal{N}(dx\, d\ell\,
du)\ind_{[0,a]}(x)u\ell f(x)+\beta\int_0^af(x)\, dx,\\
\label{def:nu_a}
\langle\nu_a,f\rangle
&=& \int\mathcal{N}(dx\, d\ell\, du)\ind
_{[0,a]}(x)\ell(1-u)f(x)+\beta\int_0^af(x)\, dx.
\end{eqnarray}
\begin{rem}
\label{rem:W}
In particular $\mu_a(dr)+\nu_a(dr)$ is defined as $\ind_{[0,a]}(r)
d_r W_r$, where $W$ is a subordinator with Laplace exponent
$\psi'-\alpha$ where $\alpha=\psi'(0)$ is defined by \reff{eq:def_alpha}.
\end{rem}

We finally set $\mathbb{M}^\psi=\int_0^{+\infty}\,da  \mathrm
{e}^{-\alpha a}\mathbb{M}_a^\psi$.

\begin{prop}[(\cite{dlgrtlpsbp},
Proposition 3.1.3)]
\label{prop:poisson_representation1}
For every nonnegative measurable function $F$ on $\cm_f(\R_+)^2$,
\[
\N^\psi \biggl[\int_0^\sigma F(\rho_t, \eta_t) \, dt
 \biggr]=\int\mathbb{M}^\psi(d\mu\,
d\nu)F (\mu, \nu) ,
\]
where $\sigma=\inf\{s>0; \rho_s=0\}$ denotes the length of the
excursion.
\end{prop}


\section{Super-critical L\'{e}vy continuum random tree}\label{sec:superCRT}
We shall construct a~L\'{e}vy CRT with super-critical branching mechanism using a Girsanov
formula.\vadjust{\goodbreak}

Let $\tilde\psi$ be a (sub)critical branching mechanism. The process
\mbox{$Z=(Z_a, a\geq0)$}, where $Z_a=L_{T_x}^a$, is a CB with branching
mechanism $\tilde\psi$. We have $\P^{\tilde\psi}_x
$-a.s. $Z_\infty=\lim_{a\rightarrow+\infty} Z_a=0$. We shall call $x$
the initial mass of the $\tilde\psi$-CRT under~$\P_x^{\tilde\psi}$.
Formula \reff{eq:defs2} readily
implies the following Girsanov's formula: for any nonnegative measurable
function $F$, and $q\geq0$,
%
%
\begin{equation}
\label{eq:Expsi}
\E^{\tilde\psi}_x [M_\infty^{\tilde\psi,q} F(\rho)  ]
=\E^{\tilde\psi_q} _x [F(\rho)  ],
\end{equation}
where $M^{\tilde\psi,q}_{\infty} $ is given by \reff{eq:Minfini}.

We will use a similar formula (with $q<0$) to define the exploration
process for a super-critical L\'{e}vy CRT with branching mechanism
$\psi$. Because super-critical branching process may have an infinite
mass, we shall cut it at a~given level to construct the corresponding
genealogical continuum random tree (see
\cite{dhpscsbp} when $\pi=0$).

For $a\geq0$, let $\cm^a_f=\cm_f([0,a])$ be the set of nonnegative
measures on $[0,a]$, and let $\cd^a$ be the set of c\`{a}d-l\`{a}g
$\cm^a_f$-valued process defined on $[0,{+\infty} )$ endowed with the
Skorohod topology. We now define a projection from $\cd$ to $\cd^a$.
For $\rho=(\rho_t, t\geq0)\in\cd$, we consider the time spent below
level $a$ up to time $t$: $\Gamma_{\rho,a}(t)=\int_0^t
\ind_{\{H(\rho_s)\leq a\}} \, ds$ and its right continuous inverse
%
%
\begin{equation}
\label{eq:def-C}
 \quad C_{\rho,a}(t)=\inf\{ r\geq0;
\Gamma_{\rho,a}(r)>t\}=\inf\biggl\{ r\geq0;
\int_0^r \ind_{\{H(\rho_s)\leq a\}} \, ds >t\biggr\},
\end{equation}
with the convention that $\inf\varnothing={+\infty} $. We
define the projector $\pi_a$ from $\cd$ to~$\cd^a$ by
%
%
\begin{equation}
\label{eq:piar}
\pi_{a}(\rho) =\bigl(\rho_{C_{\rho,a}(t)}, t\geq0\bigr),
\end{equation}
with the convention $\rho_{+\infty}=0$. By construction we have the
following compatibility relation: $\pi_a \circ\pi_b =\pi_a$ for
$0\leq
a\leq b$.

Let $\psi$ be a super-critical branching mechanism which we suppose to
be conservative, that is, \reff{eq:conservatif} holds. Recall $q^*$ is the
unique (positive) root of $\psi'(q)=0$. In particular the branching
mechanism $\psi_q$ is critical if $q=q^*$ and sub-critical if $q> q^*$.

We consider the filtration $\ch=(\ch_a, a\geq0)$ where $\ch_a$ is the
$\sigma$-field generated by the c\`{a}d-l\`{a}g process $\pi_a(\rho
)$ and the
class of $\P^{ \psi_{{q^*}}}_x$ negligible sets. Thanks to the second
statement of
Proposition \ref{prop:LT}, we get that $Z$ is $\ch$-adapted. Furthermore
the proof of Theorem 1.4.1 in \cite{dlgrtlpsbp} yields that $Z$ is
a~Markov process w.r.t. the filtration~$\ch$. In particular the process
$M^{\psi_{{q^*}},-{q^*}}$ defined by \reff{eq:defMq} is thanks to
Theorem~\ref{theo:mart} a $\ch$-martingale under $\P_x^{\psi_{q^*}}$.

Let $q\geq q^*$. We define the distribution $ \P_x^{\psi,a}$ (resp.,
$\N^{\psi, a}$) of the
$\psi$-CRT cut at level $a$ with initial mass $x$, as the
distribution of $\pi_a(\rho)$ under $M^{\psi_q,-q}_a
d\P^{\psi_q}_x$ [resp., $ \mathrm{e}^{q Z_a +\psi(q)\int_0^a Z_r\,
dr}\,d\N^{\psi_q}$]: for any
measurable nonnegative function $F$,
%
\begin{eqnarray}
\label{eq:defEpsia}
\E^{\psi,a} _x [F(\rho)  ]
&=& \E^{\psi_q} _x [M^{{\psi_q},-q}_a F(\pi_a(\rho))  ],
\\
\label{eq:defNpsia}
\N^{\psi,a}  [F(\rho)  ]
&=& \N^{\psi_q}  \bigl[\mathrm{e}^{q Z_a +\psi(q)\int_0^a Z_r\, dr}
F(\pi_a(\rho))
 \bigr].
\end{eqnarray}

\begin{lem}
\label{lem:Pyqq}
The distributions $\P^{\psi,a}_x$ and $\N^{ \psi,a }$ do not depend
on the choice of $q\geq q^*$.
\end{lem}

\begin{pf}
Let $q>q^*$. For any nonnegative measurable
function $F$, we have
\[
\E_x^{\psi_{{q}}} [M_a^{\psi_{{q}},-{q}}F(\pi_a(\rho)) ]
= \E_x^{\psi_{{q}}}\bigl [\mathrm{e}^{-{q}x+{q}Z_a+\psi({q})\int
_0^aZ_s\, ds}F(\pi_a(\rho)) \bigr].
\]
As $\psi_{{q}}=(\psi_{{q^*}})_{{q}-{q^*}}$, we apply Girsanov's formula
\reff{eq:Expsi} and the fact that $M^{\psi_{{q^*}},{q}-{q^*}}$ is a
martingale to get\vspace*{-1pt}
\begin{eqnarray*}
&&\hspace*{-3pt}\E_x^{\psi_{{q}}}
 [M_a^{\psi_{{q}},-{q}}F(\pi_a(\rho)) ]\\[-1pt]
&&\hspace*{-3pt} \qquad = \E_x^{\psi_{{q^*}}} \bigl[M_a^{\psi_{{q^*}},{q}-{q^*}}
\mathrm{e}^{-{q}x+{q}Z_a+\psi({q})\int_0^aZ_s\, ds}F(\pi_a(\rho
)) \bigr]\\
&&\hspace*{-3pt} \qquad = \E_x^{\psi_{{q^*}}} \bigl[\mathrm{e}^{({q}-{q^*}) x
-({q}-{q^*})Z_a-\psi_{{q^*}}({q}-{q^*})\int_0^a Z_s\,ds}\mathrm
{e}^{-{q}x+{q}Z_a+\psi({q})\int_0^aZ_s\, ds}F(\pi_a(\rho)) \bigr]\\
&&\hspace*{-3pt} \qquad = \E_x^{\psi_{{q^*}}}\bigl [\mathrm{e}^{-{q^*}x+{q^*} Z_a-(\psi
({q})-\psi({q^*}))\int_0^aZ_s\, ds}\mathrm{e}^{\psi({q})\int
_0^aZ_s\, ds}F(\pi_a(\rho)) \bigr]\\
&&\hspace*{-3pt} \qquad = \E_x^{\psi_{{q^*}}} [M_a^{\psi_{{q^*}},-{q^*}}F(\pi_a(\rho
)) ].
\end{eqnarray*}
Excursion theory then gives the result for the excursion measures.
\end{pf}

Let $\cw$ be the set of $\cd$-valued processes endowed with the
$\sigma$-field generated by the coordinate applications.

\begin{prop}
\label{prop:Pcomp}
Let
$(\rho^a, a\geq0)$ be the canonical process on $\cw$. There exists a
probability measure $\Pb_x^\psi$ (resp., an excursion measure
$\Nb^\psi$) on $\cw$, such that, for every $a\ge0$, the distribution
of $\rho^a$ under $\Pb_x^\psi$ (resp., $\Nb^\psi$) is~$\P^{\psi
,a} _x$
(resp., $\N^{\psi,a} $) and such that, for $0\leq a\leq b$
%
%
\begin{equation}
\label{eq:PiH}
\pi_{a} (\rho^b )=\rho^a\qquad\mbox{$\Pb_x^\psi$-a.s. (resp.,
$\Nb^\psi$-a.e.)}.
\end{equation}
\end{prop}

\begin{pf}
To prove the existence of such a projective limit, it is enough to
check the compatibility relation between $\P_x^{\psi,b}$ and
$\P_x^{\psi,a}$ for every $b\ge a\ge0$.

Let $0\leq a\leq b$. We get
\begin{eqnarray*}
\E^{\psi,b} _x [F(\pi_a(\rho))  ]
&=& \E^{\psi_{{q^*}}} _x\bigl [M^{\psi_{{q^*}},-{q^*}}_b F(\pi_a\circ
\pi_b(\rho))  \bigr]\\
&=& \E^{\psi_{{q^*}}} _x [M^{\psi_{{q^*}},-{q^*}}_b F(\pi_a(\rho
))  ]\\
&=& \E^{\psi_{{q^*}}} _x [M^{\psi_{{q^*}},-{q^*}}_a F(\pi_a(\rho
))  ]\\
&=& \E^{\psi,a} _x [F(\rho)  ],
\end{eqnarray*}
where we used the compatibility relation of the projectors for the
second equality and the fact that $M^{\psi_{{q^*}},-{q^*}}$ is a $\ch
$-martingale
for the third equality. We deduce that $\P^{\psi,b} _x \circ
\pi_a=\P^{\psi,a} _x$.

This compatibility relation implies the existence of a projective limit
$\Pb^\psi_x$. The result is similar for the excursion measure.
\end{pf}

Let us remark that the definitions of $\Pb_x^\psi$ and $\bar\N^\psi$
are also valid for a~(sub)cri\-tical branching mechanism $\psi$, with the convention
$q^*=0$. In particular, we get the following corollary.
\begin{cor}
\label{cor:Pb=P}
If $\psi$ is (sub)critical, then the law of the process
$(\pi_a(\rho),\allowbreak a\ge0)$ under $\P_x^\psi$ (resp., $\N^\psi$) is
$\Pb_x^\psi$ (resp., $\bar\N^\psi$).
\end{cor}

By construction the local time at level $a$ of
$\rho^b$ for $b\geq a$ does not depend on $b$, we denote by
$Z_a$ its value. Property (ii) of Theorem \ref{theo:mart} implies that
$Z=(Z_a,  a\geq0)$ is under $\Pb^\psi_x$ a CB with branching mechanism
$\psi$. Hence, the probability measure $\Pb_x^\psi$ can be seen as the
law of the exploration process that codes the super-critical CRT
associated with $\psi$.

We get the following direct
consequence of Properties (i) and (ii) of Lemma \ref{lem:propZ} and of the
theory of excursion measures.
\begin{cor}
\label{cor:Pbpsiq}
Let $q> 0$ such that $\psi(q)\geq0$. Then, the probability
measure $\Pb_x^{\psi_q}$ is absolutely continuous with respect to
$\Pb_x^\psi$ with
\[
\frac{d\Pb_x^{\psi_q}}{d\Pb_x^\psi}=
M^{\psi,q}_\infty=\mathrm{e}^{qx-\psi(q)\sigma}\ind_{\{\sigma
<+\infty
\}}.
\]
The measure
$\Nb^{\psi_q}$ is absolutely continuous with respect to $\Nb^\psi$
with
\[
\frac{d\Nb^{\psi_q}}{d\Nb^\psi}=\mathrm{e}^{-\psi(q) \sigma
}\ind_{\{\sigma<+\infty\}}.
\]
\end{cor}

If the total mass of $Z$, $\sigma=\int_0^{+\infty} Z_a\, da$, is
finite, then $\rho^a$ is the projection of a well-defined
exploration process.
\begin{lem}
\label{lem:sfini}
On $\{\sigma<{+\infty} \}$, there exists
$\rho^{\infty}\in\cd$ such that $\rho^a=\pi^a(\rho^\infty)$ for
all $a\geq0$, $\Pb^\psi_x$-a.s. or $ \Nb^\psi$-a.e.
\end{lem}

\begin{pf}
It is enough to get the result under $\Pb^\psi_x$.

First we assume that $\psi$ is (sub)critical.
Proposition \ref{prop:LT} implies that  \mbox{$\int_0^t \ind_{\{H(\rho
_s)\leq
a \}}\, ds$} increases to $t$ as $a$ goes to infinity. Using
\reff{eq:def-C}, \reff{eq:piar} and the right continuity of $\rho$, we
deduce that $\P^\psi_x$-a.s. for all $t\geq0$, \mbox{$
\lim_{a\rightarrow+\infty} \pi^a (\rho)_t=\rho_t$}.

Thanks to Corollary \ref{cor:Pb=P}, we deduce that
$\Pb^\psi_x$-a.s. for all $t\geq0$, $\rho^\infty
_t= \lim_{a\rightarrow+\infty} \pi^a
(\rho)_t$ exists and that $\pi_a(\rho^\infty)=\rho^a$.

The case $\psi$ super-critical is then a consequence of Corollary
\ref{cor:Pbpsiq}.\vadjust{\goodbreak}
\end{pf}

Without confusion, we shall always write $\P^\psi$ instead of $\Pb
^\psi$ and $\N^\psi$
instead of $\Nb^\psi$ and call them the law or the excursion measure
of the
exploration process of the CRT, whether $\psi$ is super-critical or
(sub)critical. And we shall write $\rho$ for the
projective limit $(\rho^a, a\geq0)$ on $\cw$, and make the
identification $\rho=\rho^\infty\in\cd$ when the latter exists,
that is,
when $\sigma$ defined by \reff{eq:defs2} is finite.

Recall $\psi^{-1}$ is given by \reff{eq:def-psi-1}. We now extend
formula \reff{eq:N_s} for general branching mechanism.
\begin{lem}
\label{lem:Nexps}
Let $\sigma$ be given by \reff{eq:defs2}. We have, for $\lambda\geq0$,
\[
\E^{\psi}_x [\mathrm{e}^{-\lambda\sigma} ]=\exp{
(- x
\N^{\psi} [1-
\mathrm{e}^{-\lambda\sigma}] )}=\mathrm{e}^{-x \psi
^{-1}(\lambda)}.
\]
\end{lem}

\begin{pf}
Let $q\geq q^*$. We have
\begin{eqnarray*}
\E^{\psi}_x [\mathrm{e}^{-\lambda\int_0^a Z_r\, dr } ]
&=& \E^{\psi_{q}}_x [M^{\psi_{{q}},-{q}}_a \mathrm{e}^{-\lambda
\int_0^a Z_r\, dr } ] \\
&=& \mathrm{e}^{-qx} \E^{\psi_{q}}_x \bigl[\mathrm{e}^{qZ_a + (\psi
(q) -\lambda) \int_0^a Z_r\, dr } \bigr] \\
&=& \mathrm{e}^{-qx} \mathrm{e}^{-x \N^{\psi_{q}} [1-\mathrm
{e}^{qZ_a + (\psi (q) -\lambda) \int_0^a Z_r\, dr } ] }\\
&=& \mathrm{e}^{-qx} \mathrm{e}^{-x \N^{\psi_{q}} [1-\mathrm
{e}^{qZ_a + \psi(q) \int_0^a Z_r\, dr } ] }\\
&&{}\times \mathrm{e}^{-x\N
^{\psi_{q}} [\mathrm{e}^{qZ_a + \psi(q) \int_0^a Z_r\, dr }
 (1- \mathrm{e}^{-\lambda \int_0^a Z_r\, dr } )  ]}
\\
&=& \E^{\psi_{q}}_x [M^{\psi_{{q}},-{q}}_a  ]
\mathrm{e}^{-x\N^{\psi} [1- \mathrm{e}^{-\lambda \int_0^a
Z_r\, dr }  ]} \\
&=& \mathrm{e}^{-x\N^{\psi} [1- \mathrm{e}^{-\lambda \int_0^a
Z_r\, dr }  ]} ,
\end{eqnarray*}
where we used \reff{eq:defEpsia} for the first equality, \reff{eq:defMq}
for the second, Lemma \ref{lem:poisson_repr} for the third,
\reff{eq:defNpsia} for the fifth and (1) of Theorem \ref{theo:mart} for
the last. We then let~$a$ goes to infinity to get the first equality of
the lemma, and use \reff{eq:IZinfini} to get the second.
\end{pf}


\section{Pruning}\label{sec:pruning}

We keep notations from Section \ref{sec:crt}. Recall that $\cd$ is
the set of
c\`{a}d-l\`{a}g $\cm_f(\R_+)$-valued process, and $\cw$ is the set of
$\cd$-valued processes. Let $R=(\rho^\theta, \theta\geq0)$ be
the canonical process on $\cw$.

Let $\psi$ be a (sub)critical branching
mechanism. The pruning procedure developed in \cite{aspsf} when
$\pi=0$, \cite{adfalp} when $\beta=0$ and in \cite{advplcrt} or
\cite{vdmfglt} for the general case, yields a probability measure on
$\cw$, $\tilde\P^\psi_x$, such that $R$ is Markov
and the law $\rho^\theta$ under $\tilde\P^\psi_x$ is $\P^{\psi
_\theta}_x$ for
all $\theta\geq0$. Furthermore $\rho^\theta$ codes for a sub-tree
of~$\rho^{\theta'}$ if $\theta\geq\theta'$. We recall the
construction of
$\tilde\P^\psi_x$ in Section \ref{sec:prune-const}.

\subsection{Pruning of (sub)critical CRT}
\label{sec:prune-const}
The main idea of the pruning procedure of a tree coded by an exploration
$\rho$ is to put marks on a~leaf $t$ (or a~branch labeled by $t$) and
more precisely on the measure $\rho_t$. There are two types of marks:
the first ones only lay on the \textit{nodes} of the tree whereas the
other ones lay on the \textit{skeleton} of the tree; each mark
appears at a random time. At time $\theta$, we remove all the
vertex of the initial tree that contains a~mark on their lineage.
In terms of exploration processes, we get $\rho^\theta$ by
a~time change of the process $\rho$ that skips all the times $t$
representing individuals that received a mark on their lineage by time
$\theta$. We explain more precisely the pruning procedure.\vspace*{-3pt}

\subsubsection{Marks on the nodes}\label{sec:Mnod}

Let $(X_t, t\geq0) $ be the L\'{e}vy process with branching mechanism
$\psi$ and let $\rho$ be the corresponding exploration
process. Recall $(\Delta_s, s\in \cj)$ denotes the set of the sizes
of jumps of $X$. Conditionally on $X$, we
consider a family
\[
(T_s,s\in\cj)
\]
of independent exponential random variables with respective
parameter
$\Delta_s.$
We define the $\cm(\R_+^2)$-valued process
$M^{(\mathrm{nod})}=(M^{(\mathrm{nod})}_t,t\ge
0)$ by
\[
M_t^{(\mathrm{nod})}(dr,dv)=\mathop{\mathop{\sum}_{{0<s\le
t}}}_{X_{s-}<I_t^s}\delta_{T_s}(dv)\delta_{H_s}(dr).
\]

For fixed $\theta\ge0$, we will consider the $\cm(\R_+)$-valued
process $M_t^{(\mathrm{nod})}(dr,[0,\theta])$ whose atoms give the marked nodes: each
node of infinite degree is marked independently from the others with probability
$1-\mathrm{e}^{-\theta\Delta_s}$, where $\Delta_s$ is the mass
(i.e., the height
of the jump) associated with the node.\vspace*{-3pt}

\begin{rem}
Although different from the measure process that defines the marks on
the nodes in \cite{adfalp} [formula (12)], this construction gives
the same marks (see  Introduction  of \cite{adfalp}).\vspace*{-3pt}
\end{rem}

\begin{rem}
The time parameter introduced here allows us to construct a coherent
family of marks. Indeed, for $\theta'>\theta$, the atoms of
$M_t^{(\mathrm{nod})}(dr,\allowbreak[0,\theta])$ are still atoms of
$M_t^{(\mathrm{nod})}(dr,[0,\theta'])$. In other words, there are more and more
marked nodes as $\theta$ increases, which allows us to construct a
``decreasing'' tree-valued process in Section \ref{sec:prunedp}.\vspace*{-3pt}
\end{rem}

\subsubsection{Marks on the skeleton}\label{sec:Mske}

Let $M^{(\mathrm{ske})}=(M_t^{(\mathrm{ske})},t\ge0)$ be a L\'{e}vy~snake with lifetime
$H$ and
spatial motion a Poisson point process with intensity
\[
2\beta\ind_{\{u>0\}}\,du.
\]
(See \cite{dlgrtlpsbp} for the definition of a L\'{e}vy snake and
\cite{advplcrt} for the extension to a~discontinuous height process
$H$; see also \cite{vdmfglt}.)

In other words, $M^{(\mathrm{ske})}$ is a $\cm(\R_+^2)$-valued process such that,
conditionally on the exploration process $\rho$:
\begin{itemize}
\item for every $t\ge0$, $M_t^{(\mathrm{ske})}(dr,du)$ is a Poisson point
measure with
intensity
\[
2\beta\ind_{[0,H_t]}(r)\,dr \ind_{\{u>0\}}\,du;\vadjust{\goodbreak}
\]
\item for every $0\le t\le t'$, with $
H_{t,t'}:=\inf_{s\in[t,t']}H_s$, then:
\begin{itemize}
\item the measures $M_t^{(\mathrm{ske})}(dr,du)\ind_{r\in[0,H_{t,t'}]}$ and
$M_{t'}^{(\mathrm{ske})}(dr,du)\ind_{r\in[0,H_{t,t'}]}$ are  equal;
\item the random measures $M_t^{(\mathrm{ske})}(dr,du)\ind_{r\in
[H_{t,t'},H_t]}$ and
$M_{t'}^{(\mathrm{ske})}(dr,du)\times \ind_{r\in[H_{t,t'},H_{t'}]}$ are independent.
\end{itemize}
\end{itemize}

\subsubsection{Definition of the pruned processes}
\label{sec:prunedp}
We define the mark process as
%
%
\begin{equation}
\label{eq:mark-process}
M^{(\mathrm{mark})}=M^{(\mathrm{nod})}+M^{(\mathrm{ske})}.
\end{equation}
The process $((\rho_t, M_t^{(\mathrm{mark})}), t\geq0)$ is called the marked
exploration process. It is Markovian (see \cite{vdmfglt} for its
properties). We
denote by $\hat\P^\psi_x$ its law and by~$\hat\N^\psi$ the
corresponding excursion measure.

For every $\theta>0$ and $t>0$, we set
\[
m_t^{(\theta)}=M_t^{(\mathrm{mark})}
 ([0,H_t]\times[0,\theta] ).
\]
The random variable $m_t^{(\theta)}$ is the number of marks at time
$\theta$ that lay on the
lineage of the individual labeled by $t$. We will only consider the
individuals without
marks on their lineage. Therefore, we set
%
%
\begin{equation}\label{eq:Ctheta}
A^{(\theta)}_t=\int_0^t\ind_{\{m_s^{(\theta)}=0\}}\,ds\quad\mbox
{and}\quad
C^{(\theta)}_t=\inf\bigl\{r\geq0; A^{(\theta)}_r\geq t\bigr\},
\end{equation}
its right-continuous inverse. Finally, we define
$\rho^\theta=(\rho^\theta_t, t\geq0)$,
$M^{(\mathrm{mark}),\theta}=(M^{(\mathrm{mark}),\theta}_t, t\geq0)$ by
\begin{eqnarray*}
\rho^\theta_t
&=& \rho_{C^{(\theta)}_t},\\
M^{(\mathrm{mark}),\theta}_t([0,h]\times
[0,q])
&=& M^{(\mathrm{mark})}_{C^{(\theta)}_t}\bigl([0,h]\times(\theta, q+\theta]\bigr).
\end{eqnarray*}

We shall use in Section \ref{sec:infinite_tree} the pruning operator
$\Lambda_\theta$ defined on the marked
exploration process by
%
%
\begin{equation}
\label{eq:notTq}
\Lambda_\theta\bigl(\rho, M^{(\mathrm{mark})}\bigr)=\bigl(\rho^\theta, M^{(\mathrm{mark}), \theta}\bigr).
\end{equation}

Using the lack of memory of the
exponential random variables and of properties of Poisson point measure,
it is easy to get
\begin{lem}
\label{lem:R-Markov}
The process $R=(\rho^\theta, \theta\geq0)$ is Markov.
\end{lem}

The $\cw$-valued process $R$ codes for a decreasing family of
CRT, which we shall call a $\psi$-family of pruned CRT.
A direct application of Theorem 1.1 of~\cite{advplcrt} gives the
marginal distribution.

\begin{prop}
\label{prop:Lrq}
The marked exploration process $(\rho^\theta, M^{(\mathrm{mark}), \theta})$
under~$\P_x^\psi$ (resp., $\N^\psi$) is distributed as $(\rho, M^{(\mathrm{mark})})$
under $\P^{\psi_\theta}_x$ (resp., $ \N^{\psi_\theta}$).\vadjust{\goodbreak}
\end{prop}

We shall now concentrate on the process $R$.
Let $\tilde\P^\psi_x$ be the law of $R$, and~$\tilde\N^\psi$ be
the
corresponding excursion measure.

We deduce the following compatibility relation from the Markov property
of $R$ and Proposition \ref{prop:Lrq}.
\begin{cor}
\label{cor:compR}
Let $\theta_0\geq0$. The law under $\tilde\P^\psi_x$ (resp.,
$\tilde
\N^\psi$) of the process $(\rho^{\theta_0+\theta},
\theta\geq0)$ is
$\tilde\P^{\psi_{\theta_0}}_x$ (resp., $\tilde\N^{\psi_\theta}$).
\end{cor}

Let us now recall the special Markov property, Theorem 4.2 of
\cite{advplcrt}, stated for the present context. We fix $\theta>0$. We
want to describe the law of the excursions of $\rho$ ``above'' the
marks, given the process ``under'' the marks. More precisely, we define
$O$ as the interior of the set $\{s\ge0, m_s^{(\theta)}=0\}$ and write
$ O=\bigcup_{i\in I}(\alpha_i,\beta_i)$. For every
$i\in
I$, we define the exploration process $\rho^{(i)}$ by: for every
$f\in\cb_+(\R_+)$, $t\ge0$,
\[
\bigl\langle
\rho^{(i)}_t,f\bigr\rangle=\int_{[H_{\alpha_i},+\infty)}f(x-H_{\alpha
_i})\rho_{(\alpha_i+t)\wedge
\beta_i}(dx).
\]
We have the following theorem.
\begin{theo}[(Special Markov property)]
\label{thm:special_Markov}
Let $\theta>0$, and let $(Z_t^\theta,t\ge0)$ be the CSBP coded by
$\rho^\theta$.
The point measure
\[
\sum_{i\in I}\delta_{(H_{\alpha_i},\rho^{(i)})}(dh,d\mu)
\]
under $\P_x^\psi$ (or $\N^\psi$) conditionally given
$(\rho_t^{\theta},t\ge0)$, is a Poisson point measure of intensity
\[
\ind_{[0,+\infty)}(h) Z_h^\theta\, dh \biggl(2\beta\theta
\N^\psi(d\mu)+\int_{(0,+\infty)}\pi(dr)(1-\mathrm{e}^{-\theta
r})\P_r^\psi(d\mu) \biggr).
\]
\end{theo}

This theorem describes in fact the joint law of
$(\rho^{(\theta)},\rho^{(\theta')})$ for $\theta<\theta'$ and
hence the transition
probabilities of the process $R$ and of the time-reversed
process. In terms of trees, by definition, the tree $\ct^{(\theta')}$
is obtained from the tree $\ct^{(\theta)}$ by pruning it with the
pruning operator $\Lambda_{\theta'-\theta}$. Conversely, to get the
tree $\ct^{(\theta)}$ from the tree $\ct^{(\theta')}$, we pick some
individuals of the tree $\ct^{(\theta')}$ according to a Poisson
point measure and add at these points either a L\'{e}vy tree
associated with the branching mechanism $\psi_\theta$ (first part of
the intensity of the Poisson measure), or an infinite node of size
$r$ and trees distributed as $\P_r^{\psi_\theta}$ (second part of
the intensity of the Poisson measure).

\subsection{Pruning of super-critical CRT}

We now use the same Girsanov techniques of Section \ref{sec:superCRT}
to define a $\psi$-family of pruned CRT when $\psi$ is
super-critical.\vadjust{\goodbreak}

Let $\psi$ be a super-critical branching mechanism which we suppose to
be conservative, that is, \reff{eq:conservatif} holds. Recall $q^*$ is the
unique (positive) root of $\psi'(q)=0$. In particular the branching
mechanism $\psi_q$ is critical if $q=q^*$ and sub-critical if $q> q^*$.

Let $q\geq q^*$. Let $R=(\rho^\theta,\theta\ge0)$ be the canonical
process on $\cw$. We set $Z=(L_\infty^a(\rho^0),a\ge0)$ which is under
$\tilde\P_x^{\psi_{{q}}}(dR)$ a CB with branching mechanism $\psi
_q$. The
process $Z$ is also well defined under the excursion measure $\tilde
\N^{\psi_q} (dR)$. We write $\pi_a(R)=(\pi_a(\rho^\theta), \theta
\geq
0)$. Notice that given the marks (i.e., given~$M^{(\mathrm{nod})}$ and
$M^{(\mathrm{ske})}$), we have $\pi_a(\rho^\theta)=(\pi_a(\rho))^\theta$.

Let $a\geq0$.
We define the distribution $ \tilde\P_x^{\psi,a}$ (resp., excursion measure
$\tilde\N^{\psi,a}$) of a $\psi$-family of
pruned CRT
cut at level $a$ with initial mass $x$, as the
distribution of $\pi_a(R)$ under $M^{\psi_q,-q}_a
d\tilde\P^{\psi_q}_x$ [resp., $ \mathrm{e}^{q Z_a +\psi(q)\int_0^a
Z_r\, dr}\,d\tilde\N^{\psi_q}$]: for any
measurable nonnegative function $F$, we have
\[
\tilde\P^{\psi,a} _x [F(R)  ]
=\tilde\P^{\psi_q} _x [M^{{\psi_q},-q}_a F(\pi_a(R))  ]
\]
and
\[
\tilde\N^{\psi,a}  [F(\rho)  ] =\tilde\N^{\psi_q}
\bigl [\mathrm{e}^{q Z_a +\psi(q)\int_0^a Z_r\, dr} F(\pi_a(\rho))
 \bigr].
\]
Same arguments as for Lemma \ref{lem:Pyqq} give the following result.
\begin{lem}
The distributions $\tilde\P^{\psi,a}_x$ and $\tilde\N^{ \psi,a
}$ do not
depend on the choice of $q\geq q^*$.
\end{lem}

As in Section \ref{sec:superCRT} (see Proposition \ref{prop:Pcomp}) the
families of measures $(\tilde\P_x^{\psi,a},x\ge
0)$ and $(\tilde\N^{\psi,a}, a\geq0) $ fulfill a compatibility
relation. Hence there exists a~projective limit $(R^a,a\ge0)$ defined
on the space of $\cw$-valued process such that:
\begin{itemize}
\item for every $a\ge0$, $R^a$ is distributed as
$\tilde\P_x^{\psi,a}$;
\item for every $a<b$, $\pi_a(R^b)=R^a$.
\end{itemize}
We write $\tilde\P_x^\psi$ for the distribution of this projective
limit and
$\tilde\N^\psi$ for the corresponding excursion measure.

By construction the local time at level $a$ of $\pi_b(\rho^\theta)$ for
$b\geq a$ does not depend on $b$, we denote by $Z_a^\theta$ its value.
Proposition \ref{prop:Lrq} and Property (ii) of Theorem~\ref{theo:mart}
imply that $Z^\theta=(Z_a^\theta, a\geq0)$ is under $\tilde\P^\psi
_x$ a CB
with branching mechanism $\psi_\theta$ started at $x$. Following
\reff{eq:defs2}, we
define $\sigma_\theta=\int_0^\infty Z^\theta_a\, da$. And, when there
is no confusion, we write $\sigma$ for $\sigma_0$.

Following Corollaries \ref{cor:Pb=P}, \ref{cor:Pbpsiq} and Lemma
\ref{lem:sfini}, we easily get the following theorem.

\begin{theo}
\label{theo:tildeP}
Let $\psi$ be a conservative branching mechanism. Let $(R^a,\allowbreak a\ge0)$
be a
$\cw$-valued process under $\tilde\P_x^\psi$ (resp., $\tilde\N
^\psi$).
\begin{longlist}[(3)]
\item[(1)]
If $\psi$ is (sub)critical, then $(R^a,a\ge0)$ under
$\tilde\P_x^\psi$ is distributed as $((\pi_a(\rho^\theta),\allowbreak  \theta
\geq0),
a\geq0) $ under $ \P_x^\psi$.\vadjust{\goodbreak}
\item[(2)] Let $q>0$ such that $\psi(q)\geq0$. Then, the probability
measure $\tilde\P_x^{\psi_q}$ is absolutely continuous with respect to
$\tilde\P_x^\psi$ with
\[
\frac{d\tilde\P_x^{\psi_q}}{d\tilde\P_x^\psi}=
M^{\psi,q}_\infty=\mathrm{e}^{qx-\psi(q)\sigma}\ind_{\{\sigma
<+\infty
\}}.
\]
The measure
$\tilde\N^{\psi_q}$ is absolutely continuous with respect to $\tilde
\N^\psi$
with
\[
\frac{d\tilde\N^{\psi_q}}{d\tilde\N^\psi}=\mathrm{e}^{-\psi(q)
\sigma}\ind_{\{\sigma<+\infty\}}.
\]
\item[(3)]
On $\{\sigma<{+\infty} \}$, there exists
$R^{\infty}\in\cw$ such that $R^a=\pi^a(R^\infty)$ for
all $a\geq0$, $\tilde\P^\psi_x$-a.s. or $ \tilde\N^\psi$-a.e.
\end{longlist}
\end{theo}

Without confusion, we shall always write $\P^\psi$ instead of $\tilde
\P^\psi$ and $\N^\psi$ instead of $\tilde\N^\psi$ and call them
the law
or the excursion measure of $\psi$-pruned family of exploration
processes, whether $\psi$ is super-critical or (sub)critical. The
$\psi$-pruned family of exploration processes codes for a $\psi$-pruned
family of continuum random sub-trees.

And we shall write $(\rho^\theta, \theta\geq0)$ for the projective
limit $(R^a, a\geq0)$, and identify it with $R^\infty\in\cw$
when the latter exists, that is, when $\sigma$ defined by~\reff{eq:defs2}
is finite. Notice that if $\sigma_\theta$ is finite, then the exploration
process $\rho^\theta$ codes for a CRT with finite mass.

\subsection{Properties of the branching mechanism}
\label{sec:propbm}
Let $\psi$ be a branching mechanism with parameter
$(\alpha,\beta,\pi)$. Let $\Theta'$ be the set of $\theta\in\R$ such
that
%
%
\begin{equation}
\label{eq:Hyp}
\int_{(1,+\infty)}\mathrm{e}^{-\theta \ell}\pi(d\ell)<+\infty.
\end{equation}
We set $\theta_\infty=\inf\Theta'$. Notice that we have either
$\Theta'=[\theta_\infty,+\infty)$ or $\Theta'=(\theta_\infty
,+\infty)$ and
that $\theta_\infty\le0$. Notice that $\psi_\theta$ exists for every
$\theta\in\Theta'$ and is conservative for every $\theta>\theta
_\infty$.
We set $\Theta=\{\theta\in\Theta'; \psi_\theta\mbox{ is
conservative}\}$. Notice that $\Theta\subset\Theta'\subset
\Theta\cup\{\theta_\infty\}$.

For instance, we have the following examples of critical branching
mechanisms:
\begin{longlist}[(iii)]
\item[(i)] quadratic case: $\psi(u)=\beta u^2$, $\Theta=\Theta'=\R$;
\item[(ii)] stable case: $\psi(u)=cu^\alpha$ with $\alpha\in(1,2)$,
$\Theta=\Theta'=[0,+\infty)$;
\item[(iii)] $\psi(u)=(u+\mathrm{e}^{-1})\log(u+\mathrm{e}^{-1})+
\mathrm{e}^{-1}$:
$\Theta=\Theta'=[-\mathrm{e}^{-1}, +\infty)$ [Notice that
$\psi_{\theta_\infty}(u)=u\log(u)$, $\psi'_{\theta_\infty
}(0+)=-\infty$ and $\psi_{\theta_\infty}$ is conservative.];
\item[(iv)] $\psi(u)=u-1+\frac{1}{1+u}$ is associated with $(\tilde
\alpha,
\beta, \pi)$ where $\tilde\alpha=2/\mathrm{e}^{}$, $\beta=0$
and $\pi(d\ell)=\mathrm{e}^{-\ell}\ind_{\{\ell>0\}}\,d\ell$:
$\Theta
=\Theta'=(-1,
+\infty)$.
\end{longlist}

For the end of this subsection, we assume that $\psi$ is CRITICAL and
that $\beta>0$ or $\pi\neq0$. Remark that $\psi$ is a one-to-one
function from $[0,+\infty)$ onto $[0,+\infty)$, and we denote by
$\psi^{-1}$ its inverse function.\vadjust{\goodbreak} For $\theta<0$ such that $\theta
\in
\Theta'$, we define $\bar\theta=\psi^{-1}(\psi(\theta))$, or,
equivalently,
$\bar\theta$ is the unique positive real number such that
%
%
\begin{equation}
\label{eq:defbt}
\psi(\bar\theta)=\psi(\theta).
\end{equation}
Since $\psi$ is continuous and strictly convex, if
$\theta_\infty\in\Theta'$, we have
%
%
\begin{equation}
\label{eq:defbqi}
\bar\theta_\infty=\lim_{\theta\downarrow\theta_\infty} \bar
\theta.
\end{equation}
Notice that in this case $\bar\theta_\infty$ is finite. If
$\theta_\infty\notin\Theta'$, we define $\bar\theta_\infty$ using
\reff{eq:defbqi}.
\begin{lem}
\label{lem:limyqi}
Let $\psi$ be CRITICAL with parameters $(\tilde\alpha, \beta, \pi)$
such that $\beta>0$ or $\pi\neq0$. If $\theta_\infty\notin
\Theta'$ then $\bar\theta_\infty=+\infty$.
\end{lem}

\begin{pf}
We assume that $\theta_\infty\notin\Theta'$. It is enough to check
that\break $\lim_{\theta\downarrow\theta_\infty}
\psi(\theta)=+\infty$ to get $\bar\theta_\infty=+\infty$.

We first consider the
case $\theta_\infty=-\infty$. Since $\psi'(0)=0$ and $\psi$ is
strictly convex, we get that $\lim_{\theta\downarrow\theta_\infty}
\psi(\theta)=+\infty$.

If $\theta_\infty>-\infty$, then using that \reff{eq:Hyp} does not
hold for $\theta_\infty$ and monotone
convergence theorem, we get that $\lim_{\theta\downarrow\theta
_\infty}
\psi(\theta)=+\infty$.
\end{pf}


\section{A tree-valued process}
\label{sec:tree-valued}
Let $\psi$ be a branching mechanism. We assume $\theta_\infty<0$.
We write $\crr_{q}=(\rho^{\gamma+q}, \gamma\geq0)$.

We deduce from Corollary \ref{cor:compR} that the families of measures
\mbox{$(\P^{\psi_\theta}, \theta\in\Theta)$} and $(\N^{\psi_\theta},
\theta\in
\Theta)$
satisfy the following compatibility property: if $\theta'<\theta$,
\mbox{$\theta'\in\Theta$}, the
process $\crr_{\theta-\theta'}$ under
$\P^{\psi_{\theta'}}$ (resp., $\N^{\psi_{\theta'}}$) is
distributed as
$\crr_0$ under $\P^{\psi_\theta}$
(resp., $\N^{\psi_\theta}$).

Hence, there exists a projective limit $\crr=(\rho^\gamma, \gamma
\in
\Theta)$ such that, for every \mbox{$\theta\in\Theta$}, the process
$(\rho^{\theta+\gamma},\gamma\ge0)$ is distributed as $(\rho
^\gamma,
\gamma\geq0)$ under $\P^{\psi_{\theta}}$. We denote by $\bP^\psi
$ the
distribution of the projective limit $\crr$, and by $\bN^\psi$ the
corresponding excursion measure. We still write $\crr_\theta$ for
$(\rho^{\theta+\gamma}, \gamma\geq0)$ for all $\theta\in\Theta$.

The process $\crr=(\rho^\theta, \theta\in\Theta)$ is Markovian, thanks
to Lemma \ref{lem:R-Markov}. It codes for a tree-valued Markov process,
which evolves according to a pruning procedure. At time $\theta$,
$\rho^\theta$ has distribution $\P^{\psi_\theta}$. Recall
$\sigma_\theta$ is the mass of the CRT coded by $\rho^\theta$. It is
not difficult to check that $\Sigma=(\sigma_\theta, \theta\in
\Theta)$
is a nonincreasing Markov process taking values in $[0,+\infty]$ and
we shall consider a version of $\crr$ such that the process $\Sigma$ is
c\`{a}d-l\`{a}g. From the continuity of $\psi$, we deduce that the
Laplace transform
of $\sigma_\theta$ given in Lemma \ref{lem:Nexps} is continuous, and
thus the process $\Sigma$ is continuous in probability.

See \cite{bfatfadbhf} for the distribution of the decreasing
rearrangement of the jumps of $(\sigma_\theta, \theta\geq0)$ in the
case of stable trees. We deduce
from the pruning procedure that a.s. $\lim_{\theta\rightarrow+\infty}
\sigma_\theta=0$. Notice that by considering the time returned process
$(\rho^{-\theta}, \theta<\theta_\infty)$, we get a Markovian
family of
exploration processes coding for a family of increasing CRTs.

\begin{rem}
\label{rem:critique}
Recall $q^*$ is the unique root of
$\psi'(q)=0$ and that $\psi_{q^*}$ is critical. Using a shift on
$\theta$ by $q^*$, that is replacing $\psi$ by $\psi_{q^*}$, one sees
that it is enough, when studying $\crr$, to assume
that $\psi$ is critical.
\end{rem}

\begin{lem}
\label{lem:Nrs}
Let $\psi$ be a critical branching mechanism with parameter
$(\alpha,\beta, \pi)$. For any $\theta\in
\Theta$, and any nonnegative measurable function $F$
defined on the state
space of $\crr_0$, we have
%
%
\begin{equation}
\label{eq:Nrs}\qquad
\bN^\psi \bigl[F(\crr_\theta)\ind_{\{\sigma_\theta<\infty\}
} \bigr]=
\bN^{\psi_\theta} \bigl[F(\crr_0)\ind_{\{\sigma_0<\infty\}
} \bigr]=
\bN^{\psi}\bigl [F(\crr_0)\mathrm{e}^{-\psi(\theta) \sigma
_0} \bigr].
\end{equation}
\end{lem}

\begin{pf}
The first equality is just the ``compatibility property'' stated at the
beginning of this section.

For $\theta\geq0$, the second equality is a direct consequence of
(ii) from Theorem~\ref{theo:tildeP}.

For $\theta<0$, let $q=\bar\theta-\theta$. Notice
that $\psi_\theta(q)=\psi(\bar\theta)-\psi(q)=0$ and\break
$(\psi_\theta)_q=\psi_{\bar\theta}$. We deduce from (ii) of Theorem
\ref{theo:tildeP} that
\[
\bN^{\psi_{\bar\theta}} [F(\crr_0) ]=
\bN^{\psi_\theta} \bigl[F(\crr_0)\ind_{\{\sigma_0<\infty\}
} \bigr].
\]
Since $\bar\theta>0$ and $\psi(\theta)=\psi(\bar\theta)$, we get from
(2) of Theorem \ref{theo:tildeP} that
\[
\bN^{\psi_{\bar
\theta}} [F(\crr_0) ]
=\bN^{\psi} \bigl[F(\crr_0)\mathrm{e}^{-\psi(\bar\theta) \sigma
_0} \bigr]=
\bN^{\psi} \bigl[F(\crr_0)\mathrm{e}^{-\psi(\theta) \sigma
_0} \bigr].
\]
This ends the proof.
\end{pf}

We deduce directly from this lemma the following result on the
conditional distribution of the exploration process knowing the total
mass of the CRT.

\begin{cor}
\label{cor:r|s}
$\!\!\!$Let $\psi$ be a branching mechanism with parameter $(\alpha,\beta,
\pi)$ such that \reff{eq:Hyp} holds. The distribution of
$(\rho^{\theta+\gamma}, \gamma\geq0)$ conditionally on
$\{\sigma_\theta=r\}$ does not depend on $\theta\in\Theta$.
\end{cor}

From this point forward, we assume that $\psi$ is CRITICAL and that
$\theta_\infty<0$.
The first assumption is not restrictive thanks to Remark \ref{rem:critique}.

Notice that $\rho^\theta$ codes for a critical (resp., sub-critical,
resp., super-critical) CRT if $\theta=0$ (resp., $\theta>0$,
resp., $\theta<0$). In particular, we have $\sigma_\theta<+\infty$
a.s. if $\theta\geq0$.

We consider the explosion time
\[
A=\inf\{\theta\in\Theta, \sigma_\theta<+\infty\},
\]
with the convention that $\inf\varnothing= \theta_\infty$.
In particular, we
have $A\leq0$ $\bP^\psi_x$-a.s. and $\bN^\psi$-a.e.
Moreover, since the process $(\sigma_\theta,\theta\in\Theta)$ is
c\`{a}d-l\`{a}g, we have, on $\{A>\theta_\infty\}$,
$\sigma_\theta=+\infty$ for every $\theta<A$ and
$\sigma_\theta<+\infty$ for every $\theta>A$.
For the time reversed process, $A$ is the random time at which the tree
gets an infinite mass.

We first give a lemma on the conditional distribution of $\sigma$.
\begin{lem}
\label{lem:sigmaq|rho}
Let $q\in\Theta$, $q\leq\theta$. We have, for
$\lambda\geq0$,
\[
\bN^\psi[\mathrm{e}^{-\lambda\sigma_q}|\rho^\theta]=\mathrm
{e}^{- \sigma _\theta \psi_\theta(\psi^{-1}_q(\lambda))}
\]
and $\bN^\psi[\sigma_q<+\infty|\rho^\theta]=\mathrm
{e}^{- \sigma_\theta \psi_\theta(\bar q-q)}$, where $\bar q=\psi
^{-1}(\psi(q))$.
\end{lem}

\begin{pf}
Let $\lambda>0$ and $F$ be a nonnegative measurable function defined on~$\cw$. We write $Z_a^q$ for the local time at level $a$ of the
exploration
process~$\rho^q$. Using \reff{eq:def-sigma}, we have
%
%
\begin{equation}
\label{eq:defI}
\bN^\psi[\mathrm{e}^{-\lambda\sigma_q}F(\rho^\theta)]=\lim
_{a\rightarrow\infty}
\bN^\psi[\mathrm{e}^{-\lambda\int_0^a Z_r^q\, dr }F(\rho^\theta)].
\end{equation}
We set
\[
I_a=\bN^\psi[\mathrm{e}^{-\lambda\int_0^a Z_r^q\, dr }F(\rho
^\theta)].
\]
Let
$G(\pi_a(\rho^\theta))=\bE^\psi[F(\rho^\theta)|\pi_a(\rho
^\theta)]$.
We
have, with $\theta'=\theta-q\geq0$,
\begin{eqnarray*}
I_a
&=& \bN^\psi[\mathrm{e}^{-\lambda\int_0^a Z_r^q\, dr }G(\pi_a(\rho
^\theta))]\\
&=& \bN^{\psi_q} [\mathrm{e}^{-\lambda\int_0^a Z_r^0\, dr }G(\pi
_a(\rho
^{\theta'}))]\\
&=& \bN^{\psi} \bigl[\mathrm{e}^{-q Z_a^0-(\psi(q)+\lambda) \int_0^a
Z_r^0\, dr}
G(\pi_a(\rho^{\theta'}))\bigr]\\
&=& \bN^{\psi} \bigl[\mathrm{e}^{-q Z_a^{\theta'} -(\psi(q)+\lambda)
\int_0^a Z_r^{\theta'} \, dr - \int_0^a K_h^a Z_h^{\theta'}\, dh }
G(\pi_a(\rho^{\theta'}))\bigr],
\end{eqnarray*}
where for the first equality we conditioned with respect to
$\sigma(\pi_a(\rho^{q}))$, used Girsanov's
formula for the third equality and Theorem \ref{thm:special_Markov} for
the last equality with
\begin{eqnarray*}
K_h^a&=&2\beta\theta'\N^{\psi}\bigl [1-\mathrm{e}^{-qZ_{a-h} -
(\psi (q)+\lambda) \int_0^{a-h} Z_r \, dr }  \bigr]\\
&&{}+ \int_{(0,+\infty)}
\pi(du)(1-\mathrm{e}^{-\theta'u})\E_u^{\psi}\bigl [1-\mathrm
{e}^{-qZ_{a-h} - (\psi(q)+\lambda) \int_0^{a-h} Z_r \, dr }
 \bigr] .
\end{eqnarray*}
We set
\begin{eqnarray*}
\tilde K_h^a&=&2\beta\theta'\N^{\psi} \bigl[\mathrm{e}^{-qZ_{a-h} -
\psi(q) \int_0^{a-h} Z_r \, dr } (1- \mathrm{e}^{-\lambda\int
_0^{a-h} Z_r\, dr })
 \bigr]\\
&&{}+ \int_{(0,+\infty)}
\pi(du)(1-\mathrm{e}^{-\theta'u})\E_u^{\psi} \bigl[\mathrm
{e}^{-qZ_{a-h} -\psi(q) \int_0^{a-h} Z_r\, dr } (1- \mathrm
{e}^{-\lambda\int_0^{a-h} Z_r \, dr })
 \bigr] .
\end{eqnarray*}
Using again Theorem \ref{thm:special_Markov} and Girsanov's formula,
we get
%
\begin{eqnarray}
\label{eq:Ia}
I_a
&=& \bN^{\psi} \bigl[\mathrm{e}^{-q Z_a^0-\psi(q)\int_0^a Z_r^0\, dr}
\mathrm{e}^{-\int_0^a (\tilde K_h^a+\lambda) Z_h^{\theta'}\, dh
}G(\pi_a(\rho^{\theta'})) \bigr]
\nonumber\\
&=& \bN^{\psi_q}\bigl [\mathrm{e}^{-\int_0^a (\tilde K_h^a+\lambda)
Z_h^{\theta'}\, dh }G(\pi_a(\rho^{\theta'}))\bigr]
\nonumber
\\[-8pt]
\\[-8pt]
&=& \bN^{\psi} \bigl[\mathrm{e}^{-\int_0^a (\tilde K_h^a+\lambda)
Z_h^{\theta}\, dh }G(\pi_a(\rho^{\theta}))\bigr]\nonumber\\
&=& \bN^{\psi} \bigl[\mathrm{e}^{-\int_0^a (\tilde K_h^a+\lambda)
Z_h^{\theta}\, dh }F(\rho^{\theta})\bigr].
\nonumber
\end{eqnarray}
Notice also that, thanks to Girsanov's formula,
\begin{eqnarray*}
\tilde K_h^a
&=& 2\beta\theta'\N^{\psi_q} [1- \mathrm{e}^{-\lambda\int
_0^{a-h} Z_r \, dr } ]\\
&&{}+ \int_{(0,+\infty)}
\pi(du)(\mathrm{e}^{-qu} -\mathrm{e}^{-\theta u})\E_u^{\psi
_q} [ 1- \mathrm{e}^{-\lambda \int_0^{a-h} Z_r \, dr } ]
\\
&=& 2\beta\theta'\bN^{\psi} [1- \mathrm{e}^{-\lambda\int
_0^{a-h} Z_r^q \, dr } ]\\
&&{}+ \int_{(0,+\infty)}
\pi(du)(\mathrm{e}^{-qu} -\mathrm{e}^{-\theta u})\bE_u^{\psi}
[ 1- \mathrm{e}^{-\lambda \int_0^{a-h} Z_r^q \, dr } ] .
\end{eqnarray*}
Using Lemma \ref{lem:Nexps}, we get
\begin{eqnarray*}
\lim_{a\rightarrow\infty} \tilde K_h^a
&=& 2\beta
\theta'\bN^{\psi} [1- \mathrm{e}^{-\lambda\sigma_q }
]+ \int
_{(0,+\infty)}
\pi(du)(\mathrm{e}^{-qu} -\mathrm{e}^{-\theta u})\bE_u^{\psi}
[ 1- \mathrm{e}^{-\lambda \sigma_q } ] \\
&=& \psi_\theta(\psi^{-1}_q(\lambda)) -
\psi_q(\psi^{-1}_q(\lambda))\\
&=& \psi_\theta(\psi^{-1}_q(\lambda))-\lambda.
\end{eqnarray*}

We deduce from \reff{eq:defI} and \reff{eq:Ia} that
\[
\bN^\psi[\mathrm{e}^{-\lambda\sigma_q}F(\rho^\theta)]=
\bN^\psi\bigl[\mathrm{e}^{-\psi_\theta(\psi^{-1}_q(\lambda)) \sigma
_\theta }F(\rho^\theta)\bigr].
\]
Letting then $\lambda$ go down to 0, we deduce, with $\bar
q=\psi^{-1}(\psi(q))$, that
\[
\bN^\psi\bigl[\ind_{\{\sigma_q<+\infty\}}F(\rho^\theta)\bigr]=
\bN^\psi\bigl[\mathrm{e}^{-\psi_\theta(\bar q-q ) \sigma_\theta
}F(\rho^\theta)\bigr].
\]
\upqed
\end{pf}

The next theorem gives the distribution of the explosion time $A$
under the measure $\bN^\psi$. Recall the definition of $\bar\theta$ in
\reff{eq:defbt} and \reff{eq:defbqi}.

\begin{theo}
\label{theo:law_A}
We have, for all $\theta\in[\theta_\infty,+\infty)$,
%
%
\begin{equation}
\label{eq:fct-rep-A}
\bN^\psi[A>\theta]=\bar\theta-\theta
\end{equation}
and
\[
\bN^\psi[A=\theta_\infty]=
\cases{\displaystyle
0 ,&\quad  if  $\theta_\infty\notin\Theta'$,\cr\displaystyle
+\infty,&\quad  if  $\theta_\infty\in\Theta'$.
}
\]
\end{theo}

\begin{pf}
We have for all $\theta>\theta_\infty$
\begin{eqnarray*}
\bN^\psi[A>\theta]
&=& \bN^\psi[\sigma_\theta=+\infty]\\
&=& \N^{\psi_\theta}[\sigma=+\infty]\\
&=& \lim_{\lambda\to
0}\N^{\psi_\theta} [1-\mathrm{e}^{-\lambda\sigma} ]\\
&=& \lim_{\lambda\to
0}\psi_\theta^{-1}(\lambda)\\
&=& \psi_\theta^{-1}(0),
\end{eqnarray*}
where we used \reff{lem:Nexps} for the fourth equality.
We get, for $t>0$,
\[
\psi_\theta(t)=0 \quad \iff \quad \psi(t+\theta)=\psi(\theta) \quad
\iff\quad
t+\theta=\bar\theta,
\]
and thus $\psi_\theta^{-1}(0)=\bar\theta-\theta$, which gives the
first part of the theorem for $\theta>\theta_\infty$. Making $\theta$
decrease to $\theta_\infty$ gives the result for $\theta_\infty$.

For the second part of the theorem, we apply the second assertion of
Lemma~\ref{lem:sigmaq|rho} with $\theta=0$.
We have, for every $q\le0$,
\[
\bN^\psi[\sigma_q<+\infty|\rho]=\mathrm{e}^{-\sigma\psi(\bar q-q)}.
\]
Then we have
\begin{eqnarray*}
\bN^\psi[A=\theta_\infty|\rho]
&=& \bN^\psi[\forall q>\theta
_\infty,
\sigma_q<+\infty|\rho]\\
&=& \lim_{q\to\theta_\infty}\bN^\psi[\sigma_q<+\infty|\rho]\\
&=& \lim_{q\to\theta_\infty}\mathrm{e}^{-\sigma\psi(\bar q-q)}\\
&=& %
\cases{\displaystyle
0 ,&\quad  if  $\theta_\infty\notin\Theta'$,\cr\displaystyle
\mathrm{e}^{-\sigma\psi(\bar\theta_\infty-\theta_\infty)} ,&\quad
 if
 $\theta_\infty\in\Theta'$,  with  $\psi(\bar\theta_\infty
-\theta_\infty)<+\infty$,
}
\end{eqnarray*}
where the last equality is a consequence of Lemma \ref{lem:limyqi}.
Then integrating with respect to $\rho$ gives the theorem.
\end{pf}

\begin{rem}\label{rem:densityA}
Since $\psi^{-1}$ is smooth, we deduce that the mapping $q\mapsto\bar
q$ is differentiable with
\[
\frac{d\bar q}{dq}=\frac{\psi'(q)}{\psi'(\bar q)}.
\]
Thus, when $\theta_\infty\notin\Theta$, we
have that the law of $A$ under $\bN^\psi$ has a density with respect to
the Lebesgue measure on $\R$ given by
\[
\ind_{\{r\in(\theta_\infty,0)\}} \biggl(1-\frac{\psi'(r)}{\psi
'(\bar
r)} \biggr).
\]
\end{rem}

\begin{theo}
\label{theo:loi-sigmaA}
 \textup{(i)}  Let $\theta\in(\theta_\infty,0)$. Under $\bN^\psi$,
conditionally on $\{A=\theta\}$,
we have for any nonnegative measurable function $F$
%
%
\begin{equation}
\label{eq:bNR|A}
\bN^\psi[F(\crr_{A})|A=\theta]=\psi'(
\bar\theta)\bN^\psi\bigl[F(\crr_0) \sigma_0 \mathrm{e}^{- \psi
(\theta) \sigma_0 }\bigr],
\end{equation}
and the law of $\sigma_A$ is given by the following: for $\lambda\geq0$
\[
\bN^\psi[\mathrm{e}^{-\lambda\sigma_A}|A=\theta]=\frac{\psi'(
\bar\theta)}{\psi'(\psi^{-1}(\lambda+\psi(\theta)))}.
\]
In particular, we have
\[
\bN^\psi[\sigma_A<\infty|A=\theta]=1.
\]

 \textup{(ii)}  If $\theta_\infty\in\Theta$, we have for any nonnegative
measurable function $F$
%
%
\begin{equation}
\bN^\psi \bigl[F(\crr_{A})\ind_{\{A=\theta_\infty\}} \bigr]=\bN
^{\psi_{\bar\theta_\infty}} [F(\crr_{0}) ].\vadjust{\goodbreak}
\end{equation}
In particular, the law of $\sigma_A$ on the event
$\{A=\theta_\infty\}$ is given by
\[
\bN^\psi \bigl[ (1-\mathrm{e}^{-\lambda\sigma_A} )\ind
_{\{
A=\theta_\infty\}} \bigr]=\psi^{-1} \bigl(\lambda+\psi(\theta
_\infty) \bigr)-\bar\theta_\infty.
\]
\end{theo}

\begin{pf}
Let $F$ be a nonnegative measurable function defined on the state space
of $\crr_0$.
Using Lemma \ref{lem:sigmaq|rho}, we get for every $\theta_\infty
<q\le\theta<0$,
\begin{eqnarray*}
\bN^\psi\bigl[F(\crr_{\theta})\ind_{\{ A> q\}}\bigr]
&=& \bN^\psi\bigl[F(\crr_{\theta}) \ind_{\{\sigma_q=+\infty\}}\bigr]\\
&=& \bN^\psi\bigl[F(\crr_{\theta})\bN^\psi[\sigma_q=+\infty|\rho
^{\theta}]\bigr]\\
&=& \bN^\psi\bigl[F(\crr_{\theta})\bigl(1- \mathrm{e}^{- \sigma_\theta \psi
_\theta(\bar q-q)}\bigr)\bigr]\\
&=& \bN^\psi\bigl[F(\crr_{\theta})\bigl(1- \mathrm{e}^{- \sigma_\theta
(\psi(\theta+\bar q-q)-\psi(\theta))}\bigr)\bigr].
\end{eqnarray*}
Thus, we get that the mapping
\[
q\mapsto\bN^\psi\bigl[F(\crr_{\theta})\ind_{\{A> q\}}\bigr]
\]
is differentiable if it is finite. As $d\bar
q/dq = \psi'(q)/\psi'(\bar q)$, we get
\begin{eqnarray*}
&&\frac{d}{dq}\bN^\psi\bigl[F(\crr_{\theta})\ind_{\{A> q\}}\bigr]\\
&& \qquad = \psi'(\bar q-q+\theta) \biggl(\frac{d\bar
q}{dq}-1 \biggr) \bN^{\psi_\theta}\bigl[F(\crr_0) \sigma_0 \mathrm
{e}^{-\sigma_0 (\psi(\bar q-q+\theta)-\psi(\theta))}\bigr]\\
&& \qquad = \psi'(\bar q-q+\theta)\frac{\psi'(q)-\psi'(\bar q)}{\psi'(\bar q)}
\bN^{\psi_\theta}\bigl[F(\crr_0) \sigma_0 \mathrm{e}^{-\sigma_0 (\psi
(\bar q-q+\theta)-\psi(\theta))}\bigr].
\end{eqnarray*}
Finally, using that $\sigma$ is right continuous, we have
\begin{eqnarray*}
\frac{\bN^\psi[F(\crr_{A}), A\in d\theta] }{d\theta}
&=& -\frac{d}{dq} \bigl(\bN^{\psi}\bigl[F(\crr_{\theta})\ind_{\{ A>
q\}} \bigr] \bigr)_{|_{q=\theta}}\\
&=& \bigl(\psi'(\bar\theta)-\psi'(\theta) \bigr) \bN^{\psi
_\theta}\bigl[F(\crr_0)
\sigma_0\ind_{\{\sigma_0<+\infty\}} \bigr] .
\end{eqnarray*}
We deduce from Lemma \ref{lem:Nrs} that
\[
\bN^\psi[F(\crr_{A})|A=\theta]
=\frac{\bN^{\psi_\theta}[F(\crr_0)
\sigma_0\ind_{\{\sigma_0<+\infty\}} ] }{\bN^{\psi_\theta}[
\sigma_0\ind_{\{\sigma_0<+\infty\}} ] }
=\frac{\bN^\psi[F(\crr_0) \sigma_0 \mathrm{e}^{- \psi(\theta)
\sigma_0 }]}
{\bN^\psi[\sigma_0 \mathrm{e}^{- \psi(\theta) \sigma_0 }]}
.
\]
This proves \reff{eq:bNR|A} but for the normalizing constant. It also
implies that
\[
\bN^\psi[\mathrm{e}^{-\lambda\sigma_A} |A=\theta]
=\frac{\N^{\psi_\theta}[\sigma\mathrm{e}^{-\lambda\sigma}]}
{\N^{\psi_\theta}[\sigma\ind_{\{ \sigma<+\infty\}} ]}.
\]
Notice that $\psi_\theta^{-1}(r)=\psi^{-1}(r+\psi(\theta))-\theta
$ for
$r\geq0$. We get from Lemma~\ref{lem:Nexps} that, for $r\geq0$,
\[
\N^{\psi_\theta}[\sigma\mathrm{e}^{-r\sigma}]=
\frac{d}{dr}\N^{\psi_\theta}[1-\mathrm{e}^{-r\sigma}]=
(\psi^{-1}_\theta)'(r)= \frac{1}{\psi'(\psi^{-1}(r+\psi(\theta
)))}.
\]
In particular, we deduce the value of the normalizing constant,
\[
\bN^\psi\bigl[\sigma_0 \mathrm{e}^{- \psi(\theta) \sigma_0 }\bigr]= \N
^{\psi_\theta}\bigl[\sigma\ind_{\{
\sigma<+\infty\}} \bigr]= 1/ \psi'( \bar\theta).\vadjust{\goodbreak}
\]
We also get
\[
\bN^\psi[\mathrm{e}^{-\lambda\sigma_A}|A=\theta]
=\frac{\psi'( \bar\theta)}{\psi'(\psi^{-1}(\lambda+\psi(\theta
)))}.
\]
This ends the proof of the first part.

For the second part of the theorem, we consider the case $\theta
_\infty
\in\Theta$.
Let us first remark that, since
the process $(\sigma_\theta,\theta\in\Theta)$ is continuous in
probability, we have
\[
\{A=\theta_\infty\}=\{\sigma_{\theta_\infty}<+\infty\}.
\]

We then apply Girsanov's formula \reff{eq:Nrs} twice to get
\begin{eqnarray*}
\bN^\psi\bigl [F(\crr_A)\ind_{\{A=\theta_\infty\}} \bigr]
&=& \bN^\psi\bigl [F(\crr_{\theta_\infty})
\ind_{\{\sigma_{\theta_\infty}<+\infty\}} \bigr]\\
&=& \bN^\psi\bigl [F(\crr_0)\mathrm{e}^{-\psi(\theta_\infty
)\sigma _0} \bigr]\\
&=& \bN^\psi \bigl[F(\crr_0) \mathrm{e}^{-\psi(\bar\theta_\infty
)\sigma_0} \bigr]\\
&=& \bN^\psi \bigl[F(\crr_{\bar\theta_\infty})
\ind_{\{\sigma_{\bar\theta_\infty}<+\infty\}} \bigr]\\
&=& \bN^{\psi_ {\bar\theta_\infty}}  [F(\crr_0) ],
\end{eqnarray*}
where we used for the last equality that
$\sigma_{\bar\theta_\infty}<+\infty$ $\bN^\psi$-a.e. and \reff{eq:Nrs}.

For $F(\crr)=1-\mathrm{e}^{-\lambda\sigma}$, we obtain
\begin{eqnarray*}
\bN^\psi \bigl[ (1-\mathrm{e}^{-\lambda\sigma_A} )\ind
_{\{
A=\theta_\infty\}} \bigr]
&=& \bN^{\psi_{\bar\theta_\infty}} [1-\mathrm{e}^{-\lambda
\sigma _0} ]\\
&=& \psi_{\bar\theta_\infty}^{-1}(\lambda)\\
&=& \psi^{-1}\bigl(\lambda+\psi(\bar\theta_\infty)\bigr)-\bar\theta_\infty.
\end{eqnarray*}
\upqed
\end{pf}

We deduce the next corollary from \reff{eq:bNR|A}.

\begin{cor}
\label{cor:r|s2}
Let $\theta_\infty<\theta<0$.
The distribution of \mbox{$\crr_A=(\rho^{A+\gamma}, \gamma\geq0)$}
conditionally
on $\{\sigma_A=r, A=\theta\}$ does not depend on
$\theta$.
\end{cor}


\section{Pruning of an infinite tree}\label{sec:infinite_tree}

We want here to define an infinite tree via a spinal description of
this tree. What we call a~spinal description of a~tree is
a~representation of the tree where a particular branch is considered (the
spine) and the subtrees that are grafted along that branch are then
described. The usual, well-known spinal descriptions of a CRT are Bismut
decomposition (see \cite{dlgpfalt}) where the spine is picked ``at
random'' among all the possible branches, and Williams decomposition
(see \cite{adwdlcrtseppnm}) where the spine is chosen to be the
highest branch of the tree. We describe next the Bismut decomposition
and show how such a~decomposition can uniquely define a~tree. Then we
define the infinite tree by such a decomposition.

\subsection{Bismut decomposition of a L\'{e}vy tree}
\label{sec:Bismut}
Let $\psi$ be a (sub)critical branching mechanism. Recall
the definition of the mark process $M^{(\mathrm{mark})}$ of\vadjust{\goodbreak} Section
\ref{sec:prunedp}. For a marked exploration process $(\rho,
M^{\mathrm{mark}})$ recall
that $\eta$ is defined by \reff{eq:eta} and notice that
$(\eta_{(\sigma-t)-}, M^{(\mathrm{mark})}_{\sigma-t}, t\in[0, \sigma])$ is
distributed as $(\rho, M^{(\mathrm{mark})})$ under the excursion measure thanks
to Corollary 3.1.6 in \cite{dlgrtlpsbp} and definition of
$M^{(\mathrm{mark})}$.

We recall that the family of pruned
exploration processes $\crr=(\rho^\theta,\theta\ge0)$ is constructed
from the exploration process $\rho$ (which is equal to $\rho^0$) and
the measure-valued process $M^{(\mathrm{mark})}$.

Let $T\ge0$. We
define under $\N^\psi$ the processes
$(\rho^{T\rightarrow},M^{(\mathrm{mark}),T\rightarrow})$ and
$(\rho^{\leftarrow T},M^{(\mathrm{mark}),\leftarrow T})$ by the following: for
every $t\ge0$,
\begin{eqnarray*}
\bigl(\rho^{T\rightarrow}_t,M^{(\mathrm{mark}),T\rightarrow}_t\bigr)
&=& \bigl(\rho
_{(T+t)\wedge\sigma},M^{(\mathrm{mark})}_{(T+t)\wedge\sigma}\bigr),\\
\bigl(\rho^{\leftarrow T}_t,M^{(\mathrm{mark}),\leftarrow T}_t\bigr)
&=& \bigl(\eta_{(T-t)\vee
0},M^{(\mathrm{mark})}_{(T-t)\vee0}\bigr),
\end{eqnarray*}
where $\rho$ is the canonical exploration process and $\eta$
its dual process.

Bismut decomposition describes in terms of Poisson point processes the former
processes when $T$ is
``uniformly distributed''
on $[0,\sigma]$.

First we must extend the definition of the measure
$\M^\psi(d\mu,d\nu)$ of \reff{def:mu_a} and \reff{def:nu_a} to
get the marks
into account. Let
\[
\cn(dx,d\ell,du)=\sum_{i\in I}\delta_{(x_i,\ell_i,u_i)}(dx,d\ell,du)
\]
be a Poisson point measure with intensity
\[
dx  \ell\pi(d\ell)  \ind_{[0,1]}(u)\,du.
\]
Conditionally on $\cn$, let $(T_i,i\in I)$ be a family of independent
exponential random variables of respective parameter
$\ell_i$. Finally, let $\tilde\cn(dk,   db)=\sum_{j\in
J}\delta_{(k_j,b_j)}(dk,  db) $ be an independent Poisson
point measure on $[0,+\infty)^2$ with intensity $2\beta dk\, db$. We
then define the spine $(\mu_a, \nu_a, m_a)$ which are three measures
given by
\begin{eqnarray*}
\mu_a(dx)
&=& \sum_{i\in
I}\ind_{[0,a]}(x_i)u_i\ell_i\delta_{x_i}(dx)+\ind_{[0,a]}(x)\beta
dx,\\
\nu_a(dx)
&=& \sum_{i\in
I}\ind_{[0,a]}(x_i)(1-u_i)\ell_i\delta_{x_i}(dx)+\ind
_{[0,a]}(x)\beta dx,\\
m_a(dx,dq)
&=& \sum_{i\in
I}\ind_{[0,a]}(x_i)\delta_{x_i}(dx)\delta_{T_i}(dq)+\sum_{j\in
J}\ind_{[0,a]}(k_j)\delta_{k_j}(dx)\delta_{b_j}(dq).
\end{eqnarray*}
We denote by $\tilde\M_a^\psi$ the law of the triple
$(\mu_a,\nu_a,m_a)$, and we set $\tilde
\M^\psi=\int_0^{+\infty}\,da\mathrm{e}^{-\psi'(0) a}\tilde\M
_a^\psi$.

Let us denote by $\P_{\mu,m}^{\psi,*}$ the law of the pair $(\rho
,M^{(\mathrm{mark})})$
starting from $(\mu,m)$ where $\rho$ is an exploration process
associated with $\psi$ and stopped when it first reaches 0.
It is easy to adapt Lemma 3.4 of
\cite{dlgpfalt} to get the following theorem.\vspace*{-2pt}

\begin{theo}[(Bismut decomposition)]
\label{theo:bismut}
For every nonnegative measurable functionals $F$ and $G$,
%
%
\begin{eqnarray}
\label{eq:bismut}
&&\N^\psi \biggl[\int_0^\sigma ds
F\bigl(\rho^{s\rightarrow},M^{(\mathrm{mark}),s\rightarrow}\bigr)G\bigl(\rho^{\leftarrow
s},M^{(\mathrm{mark}),\leftarrow
s}\bigr) \biggr]\nonumber
\\[-9pt]
\\[-9pt]
&& \qquad =\int\tilde\M^\psi(d\mu,d\nu,dm)
\E_{\mu,m}^{\psi,*}[F]\E_{\nu,m}^{\psi,*}[G].
\nonumber
\end{eqnarray}
\end{theo}

Informally speaking, the latter theorem describes a spinal
decomposition of the tree. We first pick an individual $s$
``uniformly.'' The height of that individual is ``distributed'' as
$da\mathrm{e}^{-\psi'(0)a}$. Then, conditionally on that height, the measures
$\rho_s$, $\eta_s$ and $m_s$ have law $\tilde
\M^\psi_a$. Eventually, conditionally on those measures, the marked
exploration processes on the right and on the left (reversed in time
for that one) of the individual $s$ are independent and distributed
as marked exploration processes started respectively from
$(\rho_s,m_s)$ and $(\eta_s,m_s)$, stopped when they first reach 0.

Let us now state the Poisson representation of the probability
measure~$\P_{\mu,m}^{\psi,*}$. Let $(\alpha_i,\beta_i)_{i\in I}$ be the excursion
intervals of the total mass process $(\langle\rho_t,1\rangle,\allowbreak t\ge0)$
above its minimum under $\P_{\mu,m}^{\psi,*}$. Let $(U_i,i\in I)$ be a
family of
independent random variables, independent of $\rho$ and uniformly
distributed on $[0,1]$. For every $i\in
I$, we set $x_i=H_{\alpha_i}$. Then we define $u_i$ by
\[
u_i=
\cases{\displaystyle
\rho_{\alpha_i}(\{x_i\})/\mu(\{x_i\}) ,&\quad  if   $\mu(\{x_i\})>0$,\cr\displaystyle
U_i ,&\quad  if  $\mu(\{x_i\})=0$.
}
\]
Finally, we define the measure-valued process
$\rho^{i}$ by the following: for every $t\ge0$ and every $f\in\cb
_+(\R_+)$,
\[
\langle\rho_t^{i},f\rangle=\int_{(x_i,+\infty)}
f(x-x_i)\rho_{(\alpha_i+t)\wedge\beta_i}(dx),
\]
and we define the measure valued-process $M^{(\mathrm{mark}), i}$ by the
following: for every $t\ge$ and
every $f\in\cb_+(\R_+^2)$,
\[
\bigl\langle M_t^{(\mathrm{mark}),i},f\bigr\rangle=\int_{(x_i,+\infty)\times\R
_+}f(x-x_i,\theta)
M_{(\alpha_i+t)\wedge
\beta_i}^{(\mathrm{mark})}(dx,d\theta).
\]
It is easy to adapt Lemma 4.2.4 from \cite{dlgrtlpsbp} to get the
following proposition.
\begin{prop}
\label{prop:poisson-mark}
The point measure $\sum_{i\in
I}\delta_{(x_i,u_i,\rho^i,M^{(\mathrm{mark}),i}) }$ is un-\break der~$\P_{\mu,m}^{\psi,*}$ a
Poisson point measure with intensity
\[
\mu(dx)\,du\ind_{[0,1]}(u)\N^\psi\bigl(d\rho,dM^{(\mathrm{mark})}\bigr).
\]
\end{prop}

\subsection{Reconstruction of the exploration process from a spinal
decomposition}
\label{sec:reconstruc}
Conversely, given the spinal decomposition of Bismut theorem, we reconstruct
the initial exploration process, but we must add the time\vadjust{\goodbreak} indices of
the excursions at the node (which in the previous Section are called~$u_i$).
We shall also add the mark process [see its definition
\reff{eq:mark-process}].

Let $\mu$ and $\nu$ be two finite measures such that $\operatorname{Supp}
\mu=\operatorname{Supp}  \nu=[0,H]$ and $m$ a point measure on
$[0,H]\times
\R_+$. Let $\{(\rho^i,M^{(\mathrm{mark}),i}), i\in J_g\}$ and $\{(\rho^i,
M^{(\mathrm{mark}),i}) , i\in J_d\}$ be two families of marked exploration
processes (see Section \ref{sec:prunedp}). Let $\{(x_i,u_i),i\in
J_g\cup
J_d\}$ be a family of nonnegative real numbers. The measures $\mu$ and
$\nu$ must be seen as the measures $\rho^{s\rightarrow}_0$ and~$\rho^{\leftarrow s}_0$
of Theorem \ref{theo:bismut}, the $x_i$'s are the
heights of the branching points along the chosen branch, the
$\rho^i$'s are the exploration processes that arise from the
decomposition of the processes $\rho^{s\rightarrow}$ and
$\rho^{\leftarrow s}$ above their minimum and the~$u_i$'s are
additional features that order the excursions that are attached at the
same level. The measure $m$ and the processes $M^{(\mathrm{mark}),i}$
will allow us to reconstruct the mark process.

For every $i\in J_g\cup J_d$, we set $\sigma^i$ the length of the process
$\rho^{i}$. We define
%
%
\begin{equation}
\label{eq:tau}
\cl_g=\sum_{i\in J_g}\sigma^i,\qquad\cl_d=\sum_{i\in
J_d}\sigma^i\quad\mbox{and}\quad\cl=\cl_g+\cl_d.
\end{equation}
The variable
$\cl$ represents the total length of the excursion whereas $\cl_g$
plays the
same role as $s$ in the left-hand side of Theorem \ref{theo:bismut}.
For every $i\in J_g$, we set
\[
t^i=\sum_{j\in J_g,x_j<x_i}\sigma^j+\sum_{j\in J_g,x_j=x_i\mathrm{\ and\ }u_j>u_i}\sigma^j,
\]
and, for
every $i\in J_d$, we set
\[
t^i=\cl_g+\sum_{j\in J_d,x_j>x_i}\sigma^j+\sum_{j\in
J_d,x_j=x_i\mathrm{\ and\ }u_j>u_i}\sigma^j,
\]
which is the time of the beginning of the excursion $\rho^i$.

For every $t>0$, we define the measure $\rho_t$ by
\[
\rho_t(dx)=
\cases{\displaystyle
\rho^{i}_{t-t^i}(x_i+dx)+
\mu(dx)\ind_{[0,x_i)}(x) +(u_i\nu(\{x_i\})
+\mu(\{x_i\})) \delta_{x_i}(dx)
,\cr
\displaystyle \qquad\hspace*{13pt}   \mbox{if  $t<
\cl_g, t^i\le t<t^i+\sigma^i$},\vspace*{2pt}\cr\displaystyle
\mu, \qquad   \mbox{if  $t=\cl_g$},\vspace*{2pt}\cr\displaystyle
\rho^{i}_{t-t^i}(x_i+dx)  +\mu(dx)\ind_{[0,x_i)}(x)+
u_i\mu(\{x_i\})\delta_{x_i}(dx) ,\cr\displaystyle \qquad \hspace*{13pt}   \mbox{if  $\cl_g<t< \cl,
t^i\le t<t^i+\sigma^i$,}\vspace*{2pt}\cr\displaystyle
0 , \qquad   \hspace*{2.4pt}\mbox{if  $t\geq\cl$.}
}
\]
We also define the mark process $M^{(\mathrm{mark})}(dx,dv)$ by
\[
\cases{\displaystyle
M^{(\mathrm{mark}),i}_{t-t^i}(x_i+dx,dv)+m(dx,dv)\ind_{[0,x_i]}(x)
, \cr
\displaystyle\qquad \hspace*{13.8pt}  \mbox{if  $t<
\cl_g$   or  $\cl_g<t<\cl, t^i\le t<t^i+\sigma^i$,}\cr\displaystyle
m, \qquad  \mbox{if  $t=\cl_g$,}\cr\displaystyle
0 , \qquad   \hspace*{3.2pt}\mbox{if  $t\geq\cl$.}
}
\]
We say that the process $(\rho,M^{(\mathrm{mark})})=((\rho_t,
M^{(\mathrm{mark})}_t),t\ge0)$ is
the marked exploration process
associated with the family
%
%
\begin{eqnarray}
\label{eq:def-cg}
 \cg&=&\bigl(\mu,\nu,m,\bigl(x_i, u_i,\bigl(\rho^i, M^{(\mathrm{mark}),i}\bigr),i\in
J_g\bigr),\nonumber
\\[-9pt]
\\[-9pt]
&&\hspace*{40pt}\hphantom{\bigl(}\bigl(x_i, u_i,\bigl(\rho^i, M^{(\mathrm{mark}),i}\bigr),i\in J_d\bigr) \bigr).
\nonumber
\end{eqnarray}

From Bismut decomposition, Theorem \ref{theo:bismut}, Proposition
\ref{prop:poisson-mark} and the construction of the mark process,
Section \ref{sec:prunedp}, we get the following reconstruction
corollary.

\begin{cor}
\label{cor:reconstruction}
Let $\psi$ be a (sub)critical branching mechanism. Let $(\mu, \nu,
m)$ be distributed according to $\tilde\M^\psi$. Let $\sum_{i\in
J_g} \delta_{(x_i, u_i, \rho^i, M^{(\mathrm{mark}),i})}$ and  $\sum_{i\in
J_d} \delta_{(x_i, u_i, \rho^i, M^{(\mathrm{mark}),i})}$ be conditionally on
$(\mu,\nu,m)$ independent Poisson point measures with respective
intensity
\[
\mu(dx)  \ind_{[0,1]}(u)\,du  \N^\psi\bigl(d\rho,
dM^{(\mathrm{mark})}\bigr)
\]
and
\[
\nu(dx)  \ind_{[0,1]}(u)\,du
\N^\psi\bigl(d\rho, d M^{(\mathrm{mark})}\bigr).
\]
Then the marked exploration process associated with
the family $\cg$ given by~\reff{eq:def-cg} is distributed as $(\rho,
M^{(\mathrm{mark})})$ under $\N^\psi[\sigma d(\rho,M)]$.
\end{cor}

\begin{rem}
If we start with an exploration process $\rho$, pick $s$ at random
(conditionally on $\rho$) on $[0,\sigma]$, then the decomposition of
$\rho^{s\rightarrow}$ and $\rho^{\leftarrow s}$ as excursions above
their minimum gives a family $\cg$. The exploration process $\tilde
\rho$ associated with $\cg$ given by the previous construction is not
$\rho$. Indeed, each excursion of $\tilde\rho$ ``on the left'' of $s$
is time-reversed with respect to those of $\rho$. However, the trees
coded by $\rho$ and $\tilde\rho$ are the same.\vspace*{-3pt}
\end{rem}

We can also reconstruct the pruned exploration process by pruning
$\cg$.
Let
$\theta>0$. We define the lowest mark
lying on the spine as
%
%
\begin{equation}
\label{eq:xi}
\xi_{\theta}=\sup\{x; m([0,x]\times[0,\theta])=0\}.
\end{equation}
We set $\mu^\theta=\mu\ind_{[0, \xi_\theta)}$, $\nu^\theta=\nu
\ind_{[0,
\xi_\theta)}$, $m^\theta(dx,dq)=m(dx, \theta+dq)\ind_{[0,
\xi_\theta)}(x)$, for $\delta\in\{g,
d\}$ $J_\delta^\theta=\{i\in J_\delta; x_i<\xi_\theta\}$ and
%
%
\begin{eqnarray}
\label{eq:cgq}
\cg_\theta&=&\bigl(\mu^\theta,\nu^\theta, m^\theta,
\bigl(x_i,u_i,\Lambda_{\theta}\bigl(\rho^i, M^{(\mathrm{mark}),i}\bigr),i\in
J_g^\theta\bigr),\nonumber
\\[-8pt]
\\[-8pt]
&&\hphantom{\bigl(}\hspace*{56pt}\bigl(x_i, u_i,\Lambda_{\theta} \bigl(\rho^i, M^{(\mathrm{mark}),i}\bigr),i\in
J_d^\theta\bigr) \bigr),
\nonumber
\end{eqnarray}
where the pruning operator $\Lambda_\theta$ is defined in
\reff{eq:notTq}.\vspace*{-3pt}

\begin{prop}
\label{prop:cgq}
Under the hypothesis of Corollary \ref{cor:reconstruction}, let
$(\rho^\theta,\break M^{(\mathrm{mark}), \theta})$ be the marked exploration process
associated with the family $\cg_\theta$ given
by \reff{eq:cgq}. The process $(\rho^\theta, \theta\geq0)$ is
distributed as $\crr_0$ under $\bN^\psi[\sigma_0d\crr]$.\vspace*{-3pt}
\end{prop}

\begin{pf}
Let us remark that, by construction,
$(\rho^\theta, M^{(\mathrm{mark}),\theta})=\Lambda_\theta(\rho
,\allowbreak M^{(\mathrm{mark})})$. The proposition now follows from
Corollary \ref{cor:reconstruction}.\vadjust{\goodbreak}
\end{pf}

\subsection{The infinite tree and its pruning}
\label{sec:ifp}
Let $\psi$ be a critical branching mechanism.

We build a
marked continuum random tree associated with the branching mechanism
$\psi$ using a spine decomposition with an infinite spine.
Intuitively, if
the CRT dies in finite time (which corresponds to the case $H$
continuous) this infinite CRT can be seen as the CRT conditioned to
nonextinction.

Let
\[
\cn(dx,d\ell,du)=\sum_{i\in I}\delta_{(x_i,\ell_i,u_i)}(dx,d\ell,du)
\]
be a Poisson point measure with intensity
\[
dx  \ell\pi(d\ell)  \ind_{[0,1]}(u)\,du.
\]
Conditionally on $\cn$, let $(T_i,i\in I)$ be a family of independent
exponential random variables of respective parameter
$\ell_i$. Finally, let $\tilde\cn(dk, db)=\sum_{j\in
J}\delta_{(k_j,b_j)}(dk,  db) $ be an independent Poisson
point measure on $[0,+\infty)^2$ with intensity $2\beta dk\, db$.
We define the following random measures:
\begin{eqnarray*}
\mu^*(dx)
&=& \sum_{i\in
I}u_i\ell_i\delta_{x_i}(dx)+\beta dx,\\
\nu^*(dx)
&=& \sum_{i\in
I}(1-u_i)\ell_i\delta_{x_i}(dx)+\beta dx,\\
m^*(dx,dq)
&=& \sum_{i\in
I}\delta_{x_i}(dx)\delta_{T_i}(dq)+\sum_{j\in
J}\delta_{k_j}(dx)\delta_{b_j}(dq).
\end{eqnarray*}
The measure $(\mu^*, \nu^*, m^*)$ corresponds to the the measure
$(\mu_a,
\nu_a, m_a)$ of Section~\ref{sec:Bismut}
but for an infinite spine.
Let
\[
\sum_{i\in J_g} \delta_{(x_i, u_i, \rho^i, M^{(\mathrm{mark}),i})}  \quad
\mbox{and} \quad
\sum_{i\in J_d} \delta_{(x_i, u_i, \rho^i, M^{(\mathrm{mark}),i})}
\]
be conditionally on $(\mu^*,\nu^*,m^*)$ independent Poisson point
measures with
intensity
\[
\nu^*(dx) \ind_{[0,1]}(u)\,du \N^\psi\bigl(d\rho, dM^{(\mathrm{mark})}\bigr)
\]
and
\[
\mu^*(dx) \ind_{[0,1]}(u)\,du \N^\psi\bigl(d\rho, d M^{(\mathrm{mark})}\bigr).
\]
We set
\[
\cg^*=\bigl(\mu^*,\nu^*,m^*, \bigl(x_i, u_i,\bigl(\rho^i, M^{(\mathrm{mark}),i}\bigr),i\in
J_g\bigr),\bigl(x_i, u_i,\bigl(\rho^i, M^{(\mathrm{mark}),i}\bigr),i\in J_d\bigr) \bigr),
\]
which describes the decomposition of an infinite marked tree as marked
sub-trees that are attached along its infinite spine. Let $\theta>0$.
Following the end of Section~\ref{sec:reconstruc}, we now extend the
pruning procedure to this infinite tree\vadjust{\goodbreak} by letting $\cg^*_\theta$ be
constructed from $\cg^*$ as $\cg_\theta$ given by \reff{eq:cgq} from
$\cg$ given by~\reff{eq:def-cg}
%
\begin{eqnarray}
 \xi_{\theta}^*&=&\sup\{x; m^*([0,x]\times[0,\theta])=0\}, \qquad
J_\delta^\theta=\{i\in J_\delta;
x_i<\xi_\theta^*\}\nonumber\\
\eqntext{\mbox{for } \delta\in\{g,
d\},}  \\[-4pt]
 \mu^{*,\theta}&=&\mu^*\ind_{[0,
\xi_\theta^*)},\qquad\nu^{*,\theta}=\nu^*\ind_{[0,
\xi_\theta^*)},\nonumber\\
 m^{*,\theta}(dx,dq)&=&m^*(dx, \theta+dq)\ind_{[0,
\xi_\theta^*)}(x), \nonumber\\
 \cg_\theta^*&=&\bigl(\mu^{*,\theta},\nu^{*,\theta}, m^{*,\theta},
\bigl(x_i,u_i,\Lambda_{\theta}\bigl(\rho^i, M^{(\mathrm{mark}),i}\bigr),i\in J_g^\theta\bigr),\nonumber\\
&&\hspace*{78pt}\hphantom{\bigl(}\bigl(x_i,
u_i,\Lambda_{\theta} \bigl(\rho^i, M^{(\mathrm{mark}),i}\bigr),i\in J_d^\theta\bigr) \bigr).\nonumber
\end{eqnarray}

We have the following lemma.
\begin{lem}
\label{lem:spine}
Let $\theta>0$. The probability distribution of the spine
$(\mu^{*,\theta},\nu^{*,\theta} ,\allowbreak  m^{*,\theta})$ is $ \psi
'(\theta)
\tilde\M^{\psi_\theta}$.
\end{lem}

\begin{pf}
As $\psi$ is critical, we deduce from \reff{eq:def_alpha} that
\[
\psi'(\theta)=2\beta\theta+ \int_{(0,+\infty)} (1-\mathrm
{e}^{-\theta \ell})\ell\pi(d\ell).
\]

We deduce from the theory of marked Poisson point measures that
\[
\cn^\theta(dx,d\ell,du)=\sum_{i\in
I}\ind_{\{T_i>\theta\}}\delta_{(x_i,\ell_i,u_i)}(dx,d\ell,du)
\]
is a Poisson point measure with intensity $dx  \ell\mathrm
{e}^{-\theta \ell}
\pi(d\ell)  \ind_{[0,1]}(u)\,du$. Since $\xi_\theta^*$ is independent
of $\cn^\theta$, we deduce that, conditionally on $\xi_\theta^*$,
$(\mu^{*,\theta},\nu^{*,\theta} , m^{*,\theta})$ is distributed
according to $\tilde\M^{\psi_\theta}_{\xi_\theta^*}$. Notice then
that $\xi^*_\theta$ is the
minimum of $T_1=\inf\{x_i; T_i\leq\theta, i\in I\}$ and $T_2=\inf\{k_j;
b_j\leq\theta, j\in J\}$, which are two independent exponential random
variables, which are also independent of $\cn^\theta$. The exponential
distribution of $T_1$ has parameter $\int_{(0,+\infty)}
(1-\mathrm{e}^{-\theta\ell} )\ell\pi(d\ell)$, and the
exponential distribution of $T_2$ has parameter $2\beta\theta$. Thus
$\xi^*_\theta$ has an exponential distribution with parameter
$\psi'(\theta)$, which gives the result.
\end{pf}

Let $(\rho^{\theta,*}, M^{(\mathrm{mark}), \theta,*})$ be the marked
exploration process associated with~$\cg_\theta^*$. We set
$\crr_\theta^*=(\rho^{\theta+q,*}, q\geq0)$ and denote by $\bE
^\psi$
its law. The next proposition tells us that $\crr^*_\theta$ under
$\bE^\psi$ is, up to a normalizing constant, the size biased
``distribution'' of $\crr_\theta$ under $\bN^\psi$.

\begin{prop}\label{prop:g=g*}
Let $\psi$ be a critical branching mechanism. For every positive
measurable functional $F$ and every $\theta>0$, we have
\[
\psi'(\theta)\bN^\psi[\sigma_\theta F(\crr_\theta)]
=\bE^\psi[F(\crr_\theta^*)].
\]
\end{prop}

\begin{pf}
Let $F$ be a positive measurable functional. As $\crr$ is constructed
from $(\rho, M^{(\mathrm{mark})})$, there exists a positive measurable
functional $G$ such that
\[
F(\crr)=G\bigl(\rho,M^{(\mathrm{mark})}\bigr).
\]
Moreover, there exists another positive functional $\tilde G$ such
that, for every $s\ge0$,
\[
G\bigl(\rho,M^{(\mathrm{mark})}\bigr)=\tilde
G\bigl(\bigl(\rho^{s\rightarrow},M^{(\mathrm{mark}),s\rightarrow}\bigr),\bigl(\rho^{\leftarrow
s},M^{(\mathrm{mark}), \leftarrow s}\bigr)\bigr).
\]

Then by Bismut decomposition, we have
\begin{eqnarray*}
&&\psi'(\theta) \bN^\psi[\sigma_\theta F(\crr_\theta)]\\
&& \qquad = \psi'(\theta) \bN^{\psi_\theta}[\sigma F(\crr)]\\
&& \qquad = \psi'(\theta) \bN^{\psi_\theta} \biggl[\int_0^{\sigma}
ds \tilde G\bigl(\bigl(\rho^{s\rightarrow},M^{(\mathrm{mark}),s\rightarrow}\bigr),\bigl(\rho
^{\leftarrow
s},M^{(\mathrm{mark}), \leftarrow s}\bigr)\bigr) \biggr]\\
&& \qquad = \int
\psi'(\theta) \tilde\M^{\psi_\theta}(d\mu,d\nu,dm)\E_{\mu
,m}^{\psi_\theta,*}
\otimes\E_{\nu,m}^{\psi_\theta,*}[\tilde G] .
\end{eqnarray*}
Then we conclude using Lemma
\ref{lem:spine} and the fact that $\N^{\psi_\theta}(d\rho, M^{(\mathrm{mark})})$
is the
distribution of
$\Lambda_\theta(\rho,M^{(\mathrm{mark})})$ under
$\N^\psi(d\rho,d M^{(\mathrm{mark})})$.
\end{pf}


\section{Distribution identity}\label{sec:proof_theo}
Let $\psi$ be a critical branching mechanism with
parameter $(\alpha,\beta,\pi)$. We assume that $\theta_\infty<0$.
Recall $\crr=(\crr_\theta, \theta\in\Theta)$ is defined in Section
\ref{sec:tree-valued} and $\crr^*_\theta$ in Section \ref{sec:ifp}.

\begin{theo}
\label{theo:=loi}
Let $\theta\in(\theta_\infty,0)$.
Conditionally on $\{A=\theta\}$, $\crr_A$ is distributed as
$\crr^*_{\bar\theta}$.
\end{theo}

\begin{pf}
Let $F$ be a nonnegative measurable function defined on $\cw$. We
have, for $\theta<0$,
\begin{eqnarray*}
\bN^\psi[F(\crr_A)|A=\theta]
&=& \psi'(\bar\theta)\bN^\psi\bigl[
F(\crr_0)\sigma_0\mathrm{e}^{-\psi(\theta)\sigma_0} \bigr]\\
&=& \psi'(\bar\theta)\bN^{\psi_{\bar\theta}}[\sigma_0
F(\crr_0)]\\
&=& \psi'(\bar\theta)\bN^\psi[\sigma_{\bar\theta
}F(\crr_{\bar\theta}) ]\\
&=& \bE^\psi[F(\crr^*_{\bar\theta})],
\end{eqnarray*}
where we used \reff{eq:bNR|A} for the first equality, Girsanov's formula
\reff{eq:Nrs} (with $\theta$ replaced by $\bar\theta$)
for the
second, the invariance of the distribution of $\crr$ by the shift for
the third and Proposition \ref{prop:g=g*} for the last.
\end{pf}

If $u\in(0,\bar\theta_\infty)$, let $\check u$ be the unique negative
real number such that
\[
\bar{\check u}=u.
\]
We deduce from Theorem \ref{theo:law_A} and Remark \ref{rem:densityA}
the following corollary.
\begin{cor}
\label{cor:gA}
Let us suppose that $\theta_\infty\notin\Theta$.\vadjust{\goodbreak}

Let $U$ be a positive ``random'' variable with (nonnegative) ``density''
w.r.t., the Lebesgue measure given by
\[
 \biggl(1-\frac{\psi'(r)}{\psi'(\check r)} \biggr)\ind_{\{r\in
(0,\bar\theta_\infty)\}}.
\]
Assume that $U$ is independent of $\cg^*$. Then $\crr_A$ is
distributed under $\bN^\psi$ as~$\crr^*_U$.
\end{cor}

This corollary can be viewed as a continuous analog of Proposition
26 of~\cite{aptmcdgwp}.


\section{The quadratic case}\label{sec:quadra}
We consider $\psi(\lambda)=\beta\lambda^2$ for some $\beta>0$. We have
$\Theta=\Theta'=\R$ (see the definition in Section \ref
{sec:propbm}) and
$\psi_\theta(\lambda)=\beta(\lambda^2+ 2\theta\lambda)$. Recall
$\bar\theta$ is defined by \reff{eq:def-bartheta}. So we have $\bar
\theta=|\theta|$. From Theorem \ref{theo:law_A}, we get
$\bN^\psi[A\ge\theta]=\bar\theta-\theta=2|\theta|$ for $\theta<0$
and $\bN^\psi[A\ge\theta]=0$ for $\theta\ge0$. Thus
under~$\bN^\psi$, the explosion time $A$ is distributed as 2 times the
Lebesgue measure on $(-\infty,0)$. We deduce from Theorem
\ref{theo:loi-sigmaA} the Laplace transform of the total mass of the CRT
before explosion: for $\lambda\geq0$,
\[
\bN^\psi[\mathrm{e}^{-\lambda\sigma_A}|A=\theta]=\frac{\sqrt
{\beta
\theta^2}}{
\sqrt{\lambda+ \beta\theta^2}} .
\]
In particular the distribution of $\sigma_A$ conditionally on
$\{A=\theta\}$ is the gamma distribution with parameter $(\beta
\theta^2, 1/2 )$.

Very similar computations as those in the proof of Theorem
\ref{theo:loi-sigmaA} yield that for all $s,t\geq0$, $\theta<0$,
$\lambda, \kappa\geq0$
%
%
\begin{eqnarray}
\label{eq:noyau}
&&\bN^\psi[\mathrm{e}^{-\lambda\sigma_{A+s}- \kappa\sigma_{A+s+t}
}|A=\theta]\nonumber
\\[-8pt]
\\[-8pt]
&& \qquad =\frac{\sqrt{\beta(|\theta|+s)^2}}{
\sqrt{\lambda+ \beta(|\theta|+s)^2}}\frac{\sqrt{\beta t^2} +\sqrt
{\lambda+
\beta(|\theta|+s)^2}}{
\sqrt{\kappa+(\sqrt{\beta t^2}+ \sqrt{\lambda+ \beta
(|\theta|+s)^2})^2}} .
\nonumber
\end{eqnarray}
We denote by $\sigma^*_\theta$ the total mass or length (see definition
\reff{eq:tau} of $\cl$) of the pruned infinite tree $\cg^*_\theta$.
Notice that, thanks to Proposition \ref{prop:g=g*}, $\sigma_\theta
^*$ has
the size biased distribution of $\sigma_\theta$ (the total mass of the
CRT with branching mechanism $\psi_\theta$) under $\bN^\psi$. More
precisely, we have for
any nonnegative measurable function, for $\theta>0$,
%
%
\begin{equation}
\label{eq:ss*}
2\beta\theta\bN^\psi[\sigma_\theta F(\sigma_{\theta+q}, q\geq0)]
=\bE^\psi[F(\sigma^*_{\theta+q}, q\geq0)].
\end{equation}
As the process $\Sigma=(\sigma_\theta, \theta\in\R)$ is Markov,
we get
that $\Sigma^*=(\sigma^*_\theta, \theta\geq0)$ is Markov. Notice that
a.s. $\sigma^*_0=+\infty$. Direct computations or using
\reff{eq:noyau} and Theorem \ref{theo:=loi} yield that for all
$\theta,q, \lambda, \kappa\geq0$
\[
\bE^\psi[\mathrm{e}^{-\lambda\sigma^*_{\theta}- \kappa\sigma
_{\theta+q}^* }]=\frac{\sqrt{\beta\theta^2}}{
\sqrt{\lambda+ \beta\theta^2}}\frac{\sqrt{\beta q^2} +\sqrt
{\lambda+
\beta\theta^2}}{
\sqrt{\kappa+(\sqrt{\beta q^2}+ \sqrt{\lambda+ \beta
\theta^2})^2}} .
\]
Let $\tau=(\tau_\theta, \theta\geq0)$ be the first passage process
of a
standard Brownian motion $(B_u, u\geq0)$: $\tau_\theta=\inf\{u\geq0,
B_u\geq\theta\}$. It is a stable subordinator with index $1/2$, and more
precisely with no drift, no killing and L\'{e}vy measure $(2\pi x^3)^{-1/2}\,dx$ on $(0, \infty)$: for $\lambda\geq0$, $\E[\mathrm
{e}^{-\lambda \tau_\theta}]= \mathrm{e}^{-\theta\sqrt{2\lambda
}}$. The distribution of
$\tau_\theta$ has density
\[
\frac{\theta}{\sqrt{2\pi x^3}} \mathrm{e}^{-\theta^2/2x}\ind_{\{
x>0\}}.
\]
We get the following result.
\begin{prop}
\label{prop:st}
We have:
\begin{itemize}
\item under $\bE^\psi$, $(2\beta\sigma_\theta^*, \theta\geq
0)$ is distributed as $(1/\tau_\theta, \theta\geq0)$;
\item under $\bN^\psi$, $(2\beta\sigma_{A+\theta}, \theta\geq
0)$ is distributed as $(1/(V+\tau_\theta), \theta\geq0)$ where $V$
is independent of $\tau$ and its ``distribution'' has density
w.r.t. the Lebesgue measure given by $\sqrt{2/(\pi
v)}\ind_{\{v>0\}}$.
\end{itemize}
\end{prop}

The proof of this result is postponed to the end of this section.

Notice that \reff{eq:Nrs} implies that for $\theta\geq0$,
\[
\bN^{\psi} \bigl[F(\sigma_q, q\geq0)\mathrm{e}^{-\psi(\theta)
\sigma _0} \bigr]
=\bN^{\psi} [F(\sigma_{q+\theta}, q\geq0)  ].
\]
In particular, we deduce from this, \reff{eq:ss*} and the fact that
$\tau$ is a process with independent and stationary increments the
following result (notice that the size bias effect vanish, as we
condition by $\sigma_0=1$).
\begin{cor}
\label{cor:s1} Let $\beta=1/2$.
Conditionally on $\sigma_0=1$, we have that~$(\sigma_\theta,\allowbreak
\theta\geq0)$ is under the excursion measure $\bN^\psi$ distributed
as \mbox{$(1/(1+\tau_\theta), \theta\geq0)$}.
\end{cor}

We thus recover a well-known result from Aldous and Pitman \cite{apsac}
on the size process of a tagged fragment for a self-similar
fragmentation (see \cite{brfcp}) with index $1/2$, no erosion and
binary dislocation measure $\nu$ defined on pairs
$(s_1,s_2)$ such that $s_1\ge s_2\ge0$ and $s_1+s_2=1$ by
\[
\nu(s_1\in dx)=\bigl(2\pi
x^3(1-x)^3\bigr)^{-1/2}\ind_{\{x>1/2\}}\,dx,
\]
which correspond to the
fragmentation of the CRT (see also the end of \cite{bssf,aspsf} or \cite{bfatfadbhf}).

\begin{pf*}{Proof of Proposition \ref{prop:st}}
Let $\lambda, \kappa, \theta, q$
be positive. As we did not find any reference for the computation of
\[
I=\E[\mathrm{e}^{-\lambda/\tau_\theta- \kappa/\tau_{\theta+q}}],
\]
we shall give it here. Using that $\tau$ is a subordinator, we have
\[
I=\E\bigl[\mathrm{e}^{-\lambda/\tau_\theta- \kappa/(\tau_\theta+\tau'_q)}\bigr],
\]
where $\tau'$ is an independent copy of $\tau$.
We set $p=\sqrt{2\lambda+\theta^2}$ and $
J=2\pi\frac{p}{\theta} I$. We get
\begin{eqnarray*}
J
&=& 2\pi\frac{p}{\theta} \frac{\theta q} {2\pi}
\int_{\R_+^2} \mathrm{e}^{-\lambda/x-\kappa/(x+y) - \theta^2/2x
-q^2/2y}\,
\frac{dx\,dy}{(xy)^{3/2}} \\
&=& p q \int_{\R_+^2} \mathrm{e}^{-\kappa/(x+y) - p^2/2x -q^2/2y}\,
\frac{dx\,dy}{(xy)^{3/2}}\\
&=& pq \int_{\R_+^2} \mathrm{e}^{-\kappa zu/(1+u) - zup^2/2 -zq^2/2}\,
\frac{dz\,du}{\sqrt{u}} \\
&=& p q \int_{\R_+} \frac{u+1}{u^2 p^2/2+ u(
p^2/2+ q^2/2 +\kappa) +q^2/2}\,
\frac{du}{\sqrt{u}} \\
&=& 2\gamma\int_{\R_+} \frac{u+1}{u^2 + u(
1+ \gamma^2 +\kappa' ) +\gamma^2}\,
\frac{du}{\sqrt{u}},
\end{eqnarray*}
where we used the change of variable $zu=1/x$ and $z=1/y$ for the third
equality, $\kappa'=2\kappa/p^2$ and $\gamma=q/p$ for the last.
Let $a,b$ such that $a+b=1+\gamma^2+\kappa'$ and $ab=\gamma^2$. Notice
that
\[
\frac{u+1}{u^2 + u(
1+ \gamma^2 +\kappa' ) +\gamma^2}=\frac{a-1}{a-b}\frac{1}{u+a}+
\frac{1-b}{a-b}\frac{1}{u+b}.
\]
Then we get
\begin{eqnarray*}
J
&=& 2\gamma\frac{a-1}{a-b}\int_{\R_+} \frac{du}{\sqrt{u}(u+a) }+
2\gamma\frac{1-b}{a-b}\int_{\R_+}
\frac{du}{\sqrt{u}(u+b) }\\
&=& 2\gamma\frac{1}{a-b} \biggl (\frac{a-1}{\sqrt{a}}+
\frac{1-b}{\sqrt{b}} \biggr) \int_{\R_+}
\frac{du}{\sqrt{u}(u+1) }\\
&=& 2 \gamma\frac{\sqrt{ab}+1}{\sqrt{ab}}\frac{1}{\sqrt{a}+\sqrt
{b}}\pi\\
&=& 2 \pi\frac{\gamma+1}{\sqrt{(1+\gamma)^2+\kappa'}}.
\end{eqnarray*}
Therefore, we obtain
\[
I=\frac{\theta}{p}
\frac{\gamma+1}{\sqrt{(1+\gamma)^2+\kappa'}}
= \frac{\theta}{\sqrt{\theta^2+2\lambda}}
\frac{q+\sqrt{\theta^2+2\lambda}}
{\sqrt{2\kappa+(q+\sqrt{\theta^2+2\lambda})^2}}.
\]
We deduce that the two processes, $(2\beta\sigma_\theta^*, \theta
\geq
0)$ and
$(1/\tau_\theta, \theta\geq0)$, have
the same two-dimensional marginals. Since they are Markov processes,
they have the same distribution. This proves the first part of the
theorem.

Let $U$ be a positive ``random'' variable whose ``distribution'' given
by 2
times the Lebesgue measure on $(0,+\infty)$ which is independent of
$\tau$.\vadjust{\goodbreak} The ``distribution'' of $V=\tau_U$ has density
w.r.t. the Lebesgue measure given by $\sqrt{2/(\pi
v)}\ind_{\{v>0\}}$. The second part is then a direct consequence of
Corollary~\ref{cor:gA}.
\end{pf*}


%

\printaddresses

\end{document}